\documentclass[leqno]{amsart}
\usepackage{amsmath,amsthm,amsfonts,amssymb}

\numberwithin{equation}{section}
\theoremstyle{plain}

\usepackage{graphicx}
\usepackage{subfigure}

\newtheorem{theorem}{Theorem}

\newtheorem{corollary}[theorem]{Corollary}

\newtheorem{lemma}[theorem]{Lemma}

\newtheorem{proposition}[theorem]{Proposition}

\newtheorem{remark}[theorem]{Remark}

\begin{document}

\title[A Class of Self-Exciting Point Processes]
{Asymptotics for a Class of Self-Exciting Point Processes}

\author{Tzu-Wei Yang}
\author{Lingjiong Zhu}
\address
{School of Mathematics\newline
\indent University of Minnesota-Twin Cities\newline
\indent 206 Church Street S.E.\newline
\indent Minneapolis, MN-55455\newline
\indent United States of America}
\email{
yangx953@umn.edu
\\
zhul@umn.edu}

\date{11 December 2014. \textit{Revised:} 11 December 2014}

\subjclass[2000]{60G55,60F05,60F10.}
\keywords{point processes, self-exciting processes,
limit theorems, tail probabilities.}

\begin{abstract}
In this paper, we study a class of self-exciting point processes. 
The intensity of the point process has a nonlinear dependence on the past history and time. 
When a new jump occurs, the intensity increases and we expect more jumps to come. Otherwise, the intensity
decays. The model is a marriage between stochasticity and dynamical system. 
In the short-term, stochasticity plays a major role and in the long-term, dynamical system governs the limiting
behavior of the system.
We study the law of large numbers, central limit theorem, large deviations and asymptotics for the tail probabilities.
\end{abstract}

\maketitle

\section{Introduction}

Let us consider a class of simple point process $N_{t}$ with intensity
at time $t$ given by
\begin{equation}\label{Dynamics}
\lambda_{t}:=\lambda\left(\frac{N_{t-}+\gamma}{t+1}\right),
\end{equation}
where $\lambda(\cdot):\mathbb{R}_{\geq 0}\rightarrow\mathbb{R}_{\geq 0}$ is a non-negative function
and we also assume that the point process has an empty past history, i.e., $N(-\infty,0]=0$.
$\gamma\geq 0$ will be called the initial condition and notice that $\lambda_{0}=\lambda(\gamma)$. 
We use $\frac{N_{t-}}{t+1}$ instead of $\frac{N_{t-}}{t}$ to
avoid the singularity at $t=0$. We use $N_{t-}$ instead of $N_{t}$ in \eqref{Dynamics} to guanrantee that
the intensity if $\mathcal{F}_{t}$-predictable, where $\mathcal{F}_{t}$ is the natural filtration.

The simple point process $N_{t}$ by its definition, is represents
a wide class of self-exciting point processes. When $x\mapsto\lambda(x)$ is an increasing function, the intensity
$\lambda_{t}$ increases whenever there is a new jump and otherwise it decays. This pheonomenon is known
as the self-exciting property in the literature. Self-exciting processes have been widely 
studied in the literature. 
The self-exciting property makes it ideal
to characterize the correlations in some complex systems, including finance.
Bacry et al. \cite{Bacry}, Bacry et al. \cite{BacryII} studied microstructure noise and Epps effect;
Chavez-Demoulin et al. \cite{Chavez} studied value-at-risk;
Errais et al. \cite{Errais} used self-exciting affine point processes to model the credit risk. 
A Cox-Ingersoll-Ross process with self-exciting jumps is proposed to model the short rate
in interest rate models in Zhu \cite{ZhuCIR}. 
 
The self-exciting point processes have also been applied to other fields, including 
seismology, see e.g. Hawkes and Adamopoulos \cite{HawkesIV}, Ogata \cite{OgataII}, 
sociology, see e.g. Crane and Sornette \cite{Crane} and Blundell et al. \cite{Blundell},
and neuroscience, see e.g. Chornoboy et al. \cite{Chornoboy}, Pernice et al. \cite{PerniceI},
Pernice et al. \cite{PerniceII}.

The most popular class of self-exciting point processes is Hawkes process, introduced by Hawkes \cite{Hawkes}. 
It has birth-immigration respresentation, see Hawkes and Oakes \cite{HawkesII} and the limit theorems
and Bartlett spectrum have been well studied in the literature, see e.g. Hawkes \cite{Hawkes}, Bacry et al. \cite{Bacry},
Bordenave and Torrisi \cite{Bordenave}, Zhu \cite{ZhuMDP}. The limit theorems for some variations and extensions 
of the linear Hawkes processes have been studied in e.g. Karabash and Zhu \cite{Karabash}, Zhu \cite{ZhuCIR}, 
Fierro et al. \cite{Fierro}, Merhdad and Zhu \cite{MehrdadZhu}.

Br\'{e}maud and Massouli\'{e} \cite{Bremaud} introduced
nonlinear Hawkes process as a generalization of classical Hawkes process. It is a simple point process with intensity
\begin{equation}
\lambda_{t}=\lambda\left(\int_{(-\infty,0)}h(t-s)N(ds)\right),
\end{equation}
where $\lambda(\cdot),h(\cdot)$ satisfies certain conditions and many properties of this process, and in particular 
the limit theorems have been studied recently in Zhu \cite{ZhuI}, Zhu \cite{ZhuII} and Zhu \cite{ZhuIII}.
The name ``nonlinear'' comes from the nonlinearity of the function $\lambda(\cdot)$. When $\lambda(\cdot)$
is linear, it reduces to the classical Hawkes process. Unlike linear Hawkes process, the nonlinear Hawkes process
does not lead to closed-form formulas of the limiting mean, variance in law of large numbers, central limit theorems
and the rate function in large deviations. 

The model \eqref{Dynamics} proposed in this paper preserves the self-exciting property and nonlinear structure
of the nonlinear Hawkes processes while at the same time have more analytical tractability. 
We also note that if we replace $\lambda_{t}$ in \eqref{Dynamics} by $\lambda(N_{t-})$, it becomes the classical
pure-birth process, see e.g. Feller \cite{Feller}. 

The model \eqref{Dynamics} is time-inhomogeneous Markovian. This can be seen by letting $Y_{t}:=\frac{N_{t}}{t+1}$ and $Y_{t}$ satisfies
\begin{equation}
dY_{t}=\frac{\lambda(Y_{t-})-Y_{t}}{t+1}dt+\frac{dM_{t}}{t+1},
\end{equation}
where $M_{t}=N_{t}-\int_{0}^{t}\lambda_{s}ds$ is a martingale. Let us define $\bar{Y}_{t}$ as the deterministic solution of
\begin{equation}
d\bar{Y}_{t}=\frac{\lambda(\bar{Y}_{t})-\bar{Y}_{t}}{t+1}dt.
\end{equation}
The limiting behavior of $\bar{Y}_{t}$ is well understood in dynamical systems. Under the assumptions that
there are finitely many fixed points of $x=\lambda(x)$. If we order these fixed points as $x_{1}<x_{2}<\cdots<x_{K}$, 
then $x_{1},x_{3},x_{5},\ldots$ are stable fixed points and $x_{2}<x_{4}<\cdots$ are unstable fixed points.
If $\bar{Y}_{0}$ lies on any one of the fixed points, then $\bar{Y}_{t}$ stays there. Otherwise, $\bar{Y}_{0}$ must lie
between a stable fixed point and an unstable fixed point and $\bar{Y}_{t}$ will converge to that neighboring stable fixed point
as $t\rightarrow\infty$. 
This is not true in our model \eqref{Dynamics}. For example, if $Y_{t}$ starts at $Y_{0}$ between $x_{1}$ and $x_{2}$, it may not
end up at $x_{1}$ as $t\rightarrow\infty$. That is because there exists a positive probability that the process can jump above $x_{2}$.
However, as time $t$ becomes large, the jump size of $Y_{t}$ that is $\frac{1}{t+1}$ becomes small. Therefore,
when time $t$ is small, stochasticity plays a major role in the behavior of \eqref{Dynamics} and when time $t$ is large, the
behavior of \eqref{Dynamics} is governed by the dynamical system.
Hence our model \eqref{Dynamics} can be seen as a marriage between the dynamical system and stochasticity.

Here are a list of questions we are interested to study.
\begin{itemize}
\item
If the equation $x\mapsto\lambda(x)$ has a unique fixed point $x^{\ast}$, do we have $\frac{N_{t}}{t}\rightarrow x^{\ast}$?

\item
What if the equation $x=\lambda(x)$ has more than one solution?
What should be the limiting set of $\frac{N_{t}}{t}$ as $t\rightarrow\infty$ then?

\item
What if $x=\lambda(x)$ has no solutions, what should be the correct scaling for $N_{t}$? 
And what should be the correct scalings for $\mathbb{E}[N_{t}]$ and $\text{Var}[N_{t}]$?

\item
Can we study the large deviations for $\mathbb{P}(\frac{N_{t}}{t}\in\cdot)$? And central limit theorems?

\item
For a fixed time interval $[0,t]$, what is the asymptotics for the tail probabilities $\mathbb{P}(N_{t}\geq\ell)$
as $\ell\rightarrow\infty$?
\end{itemize}

We will show that under certain conditions for the model \eqref{Dynamics}, 
the limiting sets of the law of large numbers for $\frac{N_{t}}{t}$ is the 
set of stable fixed points of $x=\lambda(x)$. When the equation $x=\lambda(x)$ has a unique fixed point, the limit
is therefore the unique fixed point. It gets more interesting when there are more than one fixed point.
Second-order properties will also been studied, including the variance and covariance structure. A sample-path
large deviation principle will be derived and hence the large deviations for $\mathbb{P}(N_{t}/t\in\cdot)$ as well.

\begin{figure}
\centering
\begin{subfigure}
  \centering
  \includegraphics[scale=0.40]{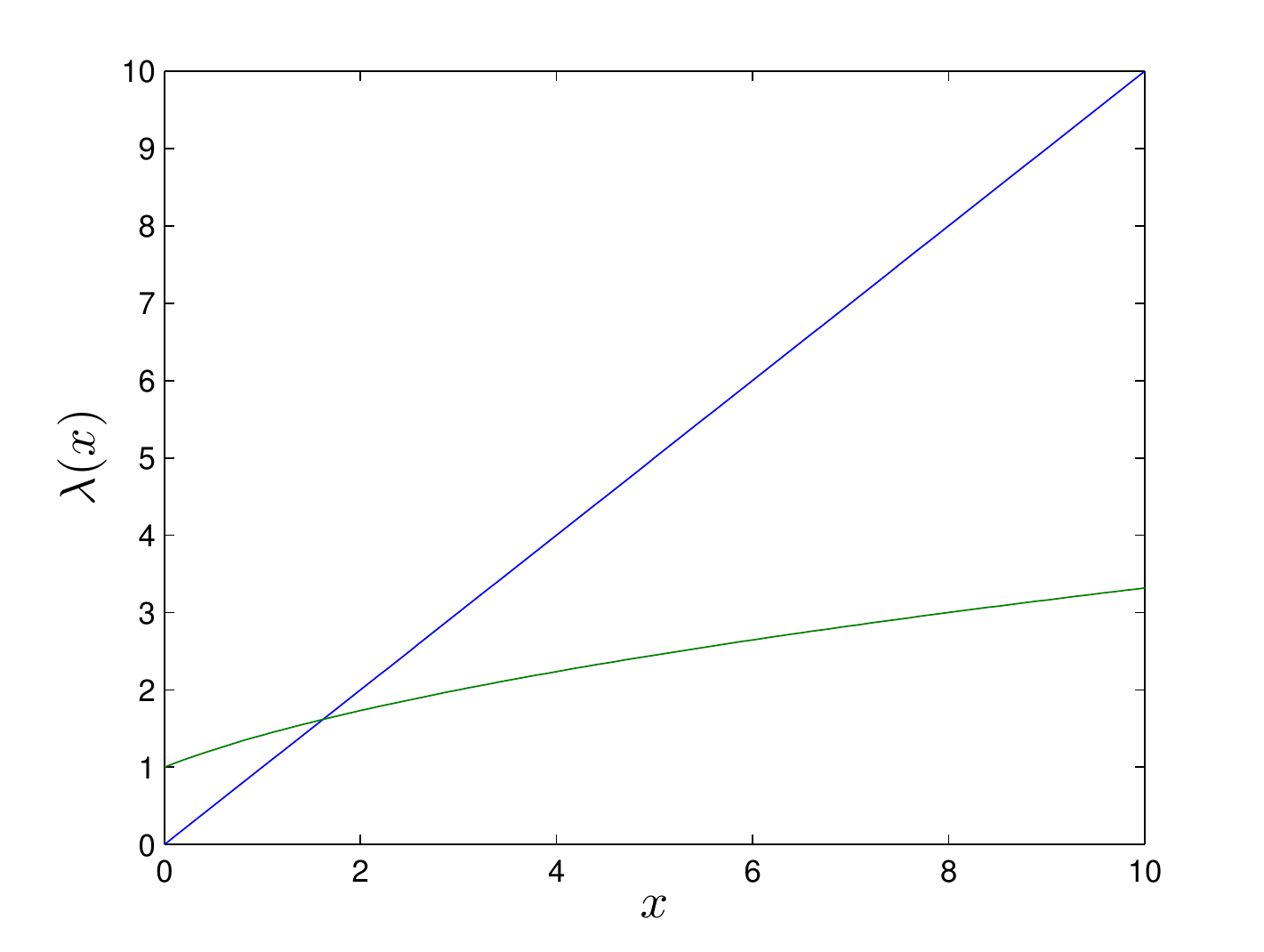}
\end{subfigure}
\begin{subfigure}
  \centering
  \includegraphics[scale=0.40]{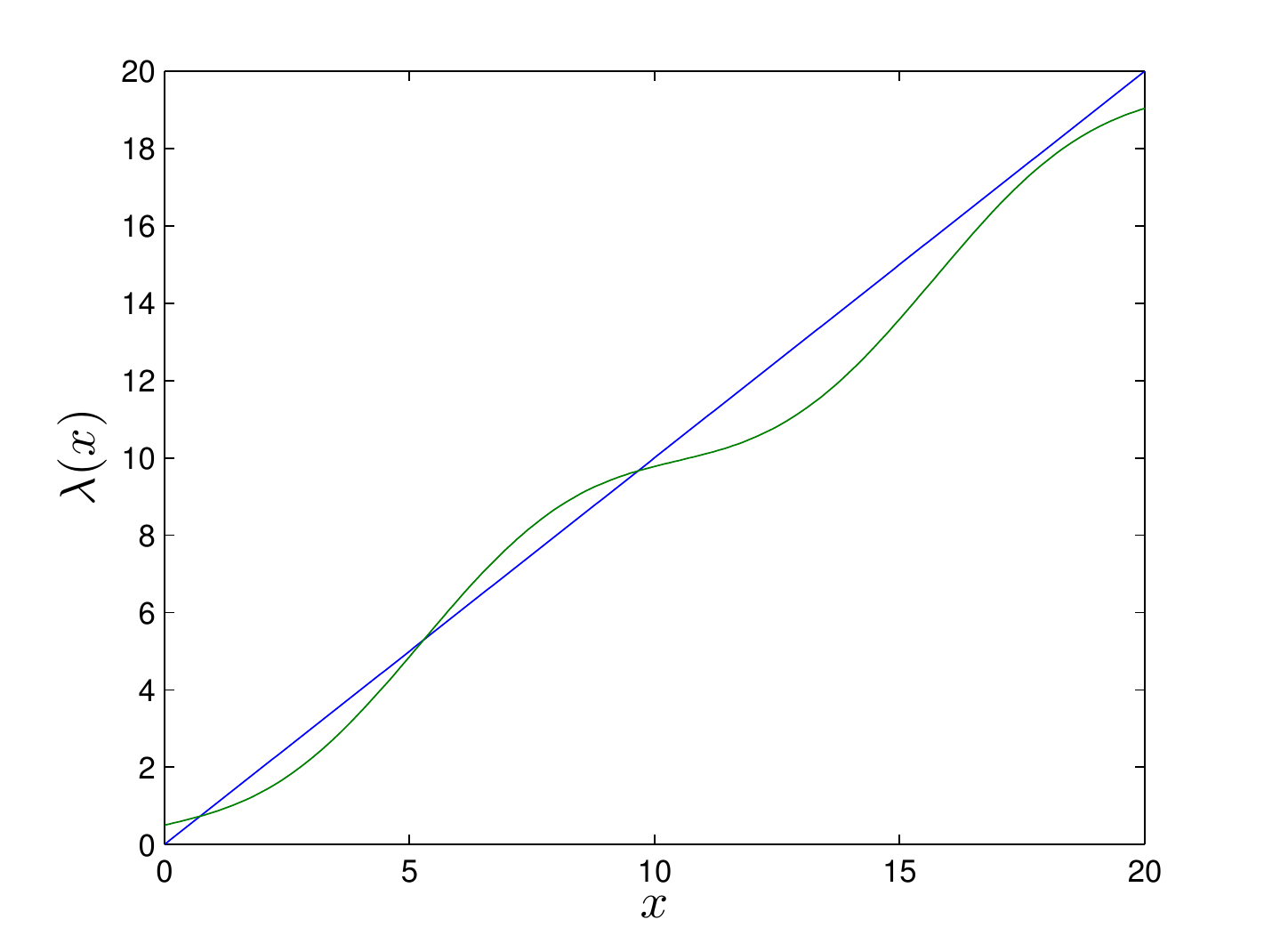}
\end{subfigure}
\caption{On the left hand side, we have the plot of $\lambda(x)=\sqrt{1+x}$ 
and it has a unqiue fixed point to
the equation $x=\lambda(x)$, $x^{\ast}=\frac{1+\sqrt{5}}{2}$. On the right hand side,
we have the plot of $\lambda(x)=0.9x-\sin(0.6x)+0.5$ and there are three fixed points
of the equation $x=\lambda(x)$. Two are stable and the one in between is unstable.}
\end{figure}

\begin{figure}
\centering
\begin{subfigure}
  \centering
  \includegraphics[scale=0.35]{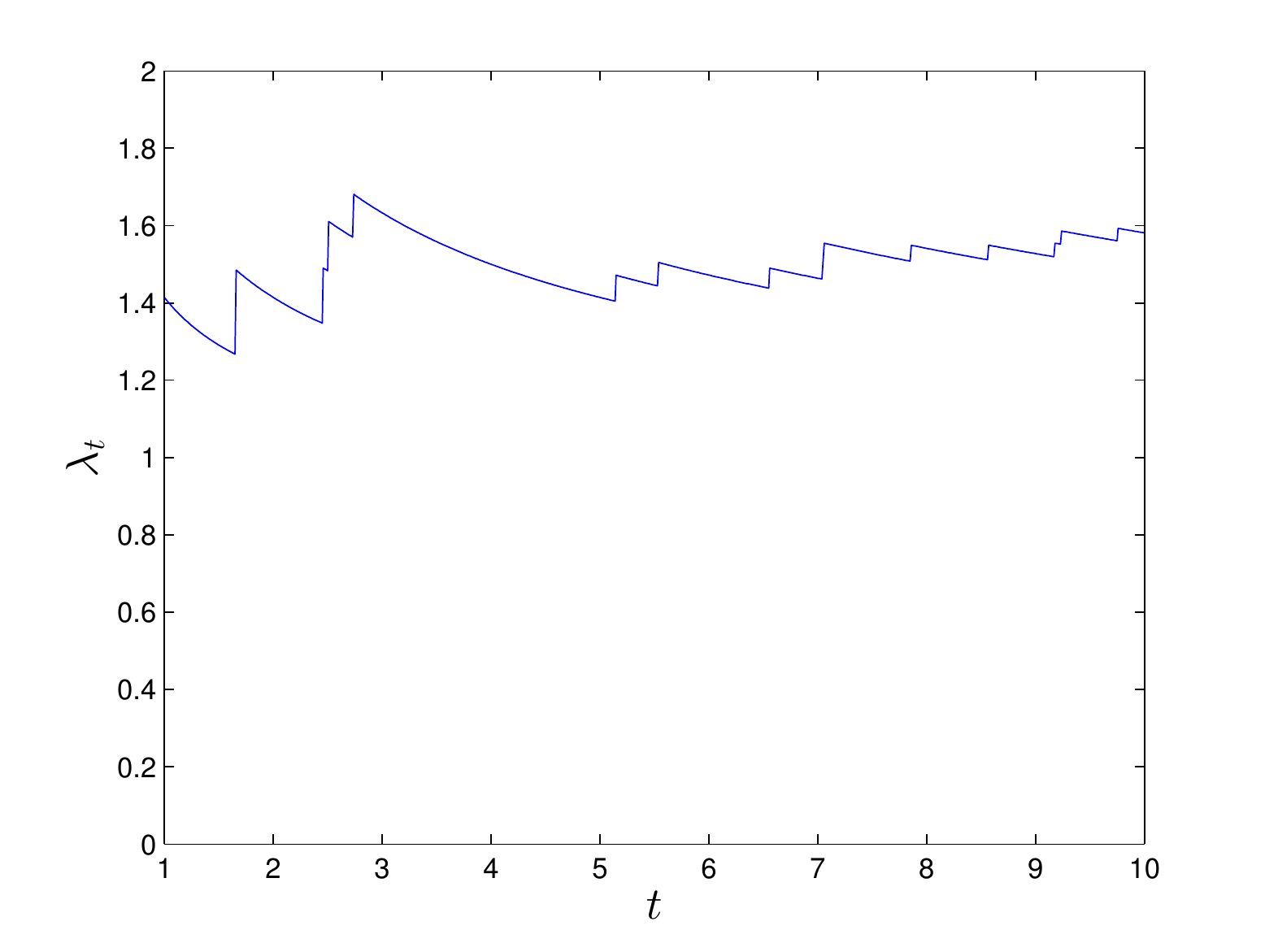}
\end{subfigure}
\begin{subfigure}
  \centering
  \includegraphics[scale=0.35]{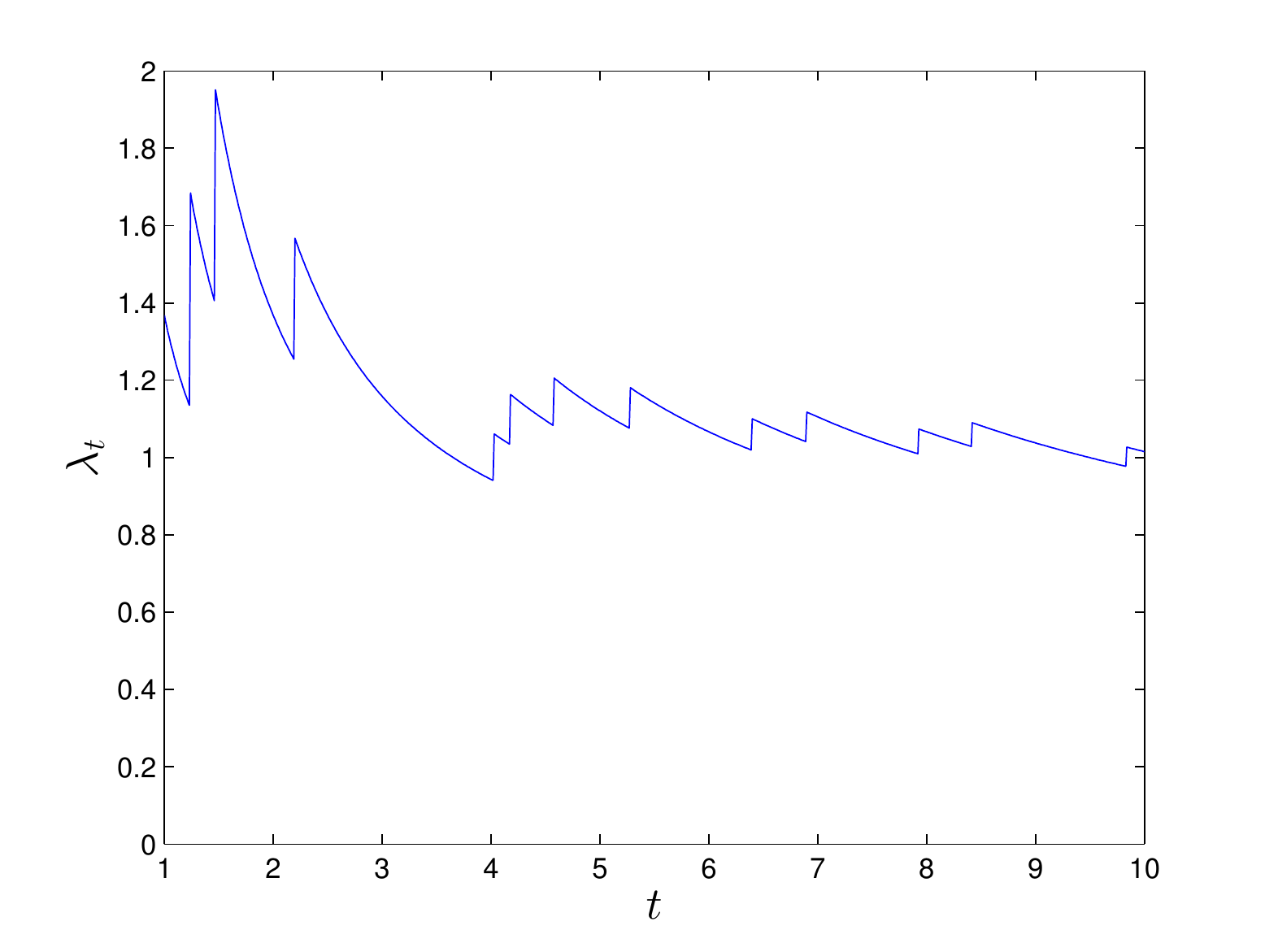}
\end{subfigure}
\caption{On the left hand side, we have the plot of $\lambda_{t}$ for $\lambda(x)=\sqrt{1+x}$. 
On the right hand side,
we have the plot of $\lambda_{t}$ for $\lambda(x)=0.9x-\sin(0.6x)+0.5$. We zoom into the time interval
$[1,10]$ to see the local self-exciting behavior of the model.}
\end{figure}

\begin{figure}[htb]
\begin{center}
\includegraphics[scale=0.70]{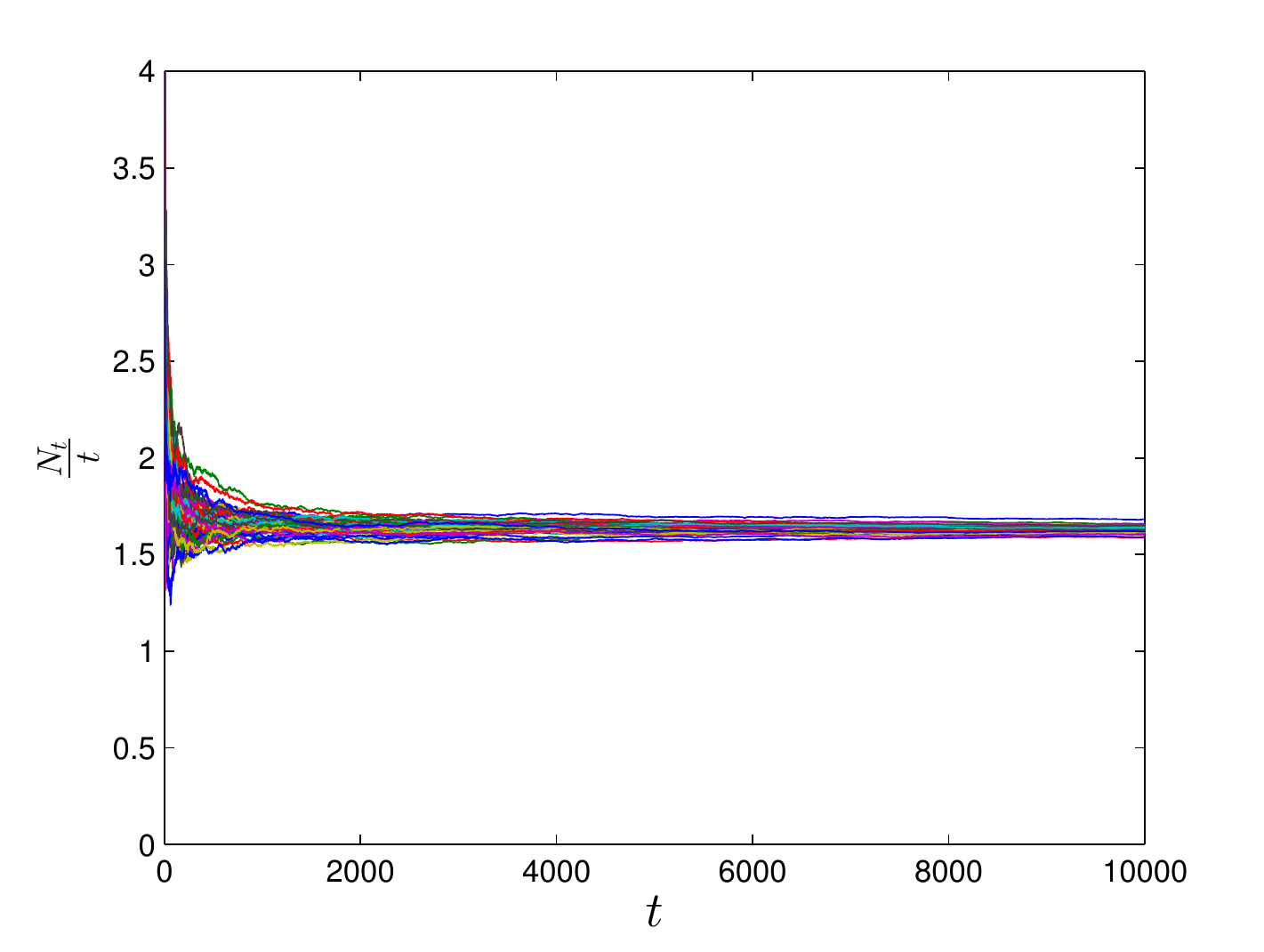}
\caption{$\lambda(x)=\sqrt{1+x}$ and it has a unqiue fixed point to
the equation $x=\lambda(x)$, $x^{\ast}=\frac{1+\sqrt{5}}{2}$. 
The initial condition is assumed to be $N_{0}=5$.
As time $t\rightarrow\infty$,
$\frac{N_{t}}{t}$ converges to this unique fixed point.}
\label{Unique}
\end{center}
\end{figure}

\begin{figure}[htb]
\begin{center}
\includegraphics[scale=0.70]{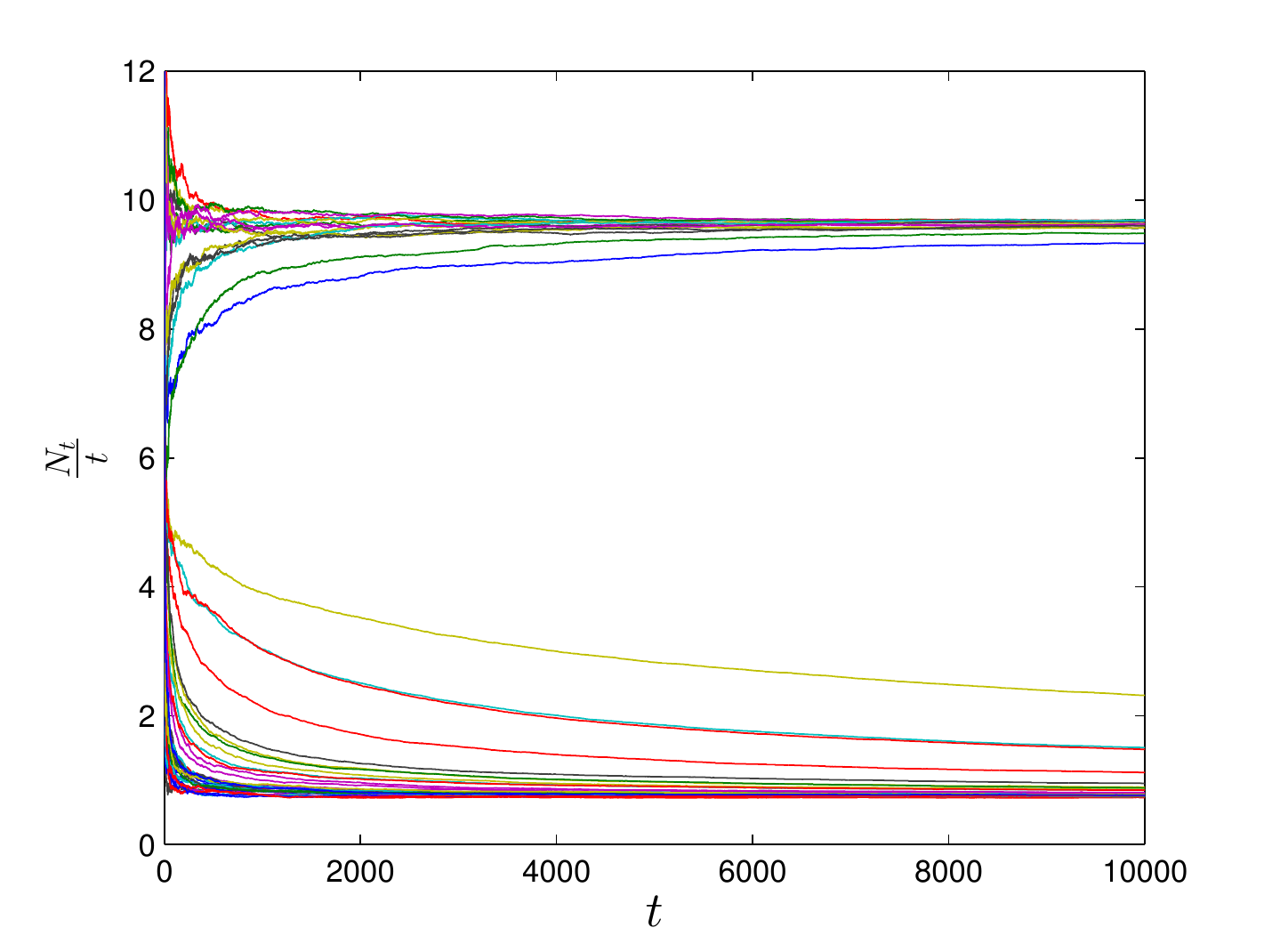}
\caption{$\lambda(x)=0.9x-\sin(0.6x)+0.5$ and there are three fixed points
of the equation $x=\lambda(x)$. The initial condition is assumed to be $N_{0}=5$.
Two are stable and the one in between is unstable.
As time $t\rightarrow\infty$, $\frac{N_{t}}{t}$ will converge to either one of the stable fixed points.}
\label{Nonunique}
\end{center}
\end{figure}

Figure \ref{Unique} illustrates that when there is a unique fixed point of $x=\lambda(x)$, 
as time $t\rightarrow\infty$, $\frac{N_{t}}{t}$ converges to this unique fixed point.
When there are more than one fixed point, Figure \ref{Nonunique} illustrates that as time $t\rightarrow\infty$, $\frac{N_{t}}{t}$
will converge to the set of all the stable fixed points of $x=\lambda(x)$. Let us say in Figure \ref{Nonunique}, the two stable
fixed points are $x_{1}<x_{2}$. Let
\begin{equation}
p_{1}(x):=\mathbb{P}_{N_{0}=x}\left(\lim_{t\rightarrow\infty}\frac{N_{t}}{t}=x_{1}\right),
\qquad
p_{2}(x):=\mathbb{P}_{N_{0}=x}\left(\lim_{t\rightarrow\infty}\frac{N_{t}}{t}=x_{2}\right).
\end{equation}
Then, for any initial condition $N_{0}=x$, $p_{1}(x)+p_{2}(x)=1$. Intuitively, it is clear that when $N_{0}=x$ 
is closer to $x_{1}$ than $x_{2}$ is in between $x_{1}$ and $x_{2}$, 
there should be a higher probability for the limit $\lim_{t\rightarrow\infty}\frac{N_{t}}{t}$ to end up
at $x_{1}$. If the starting point is lower than $x_{1}$, it is also more likely for the limit to end up at $x_{1}$ and so on and so forth.
We can therefore use the same $\lambda(x)$ as in Figure \ref{Nonunique} and make a plot of $p_{1}$ and $p_{2}$
as a function of the initial starting point $N_{0}=x$. From Figure \ref{p1p2}, it turns out that $p_{1}(x)$ is monotonically decreasing
in $x$ and $p_{2}(x)=1-p_{1}(x)$ is monotonically increasing in $x$.

\begin{figure}[htb]
\begin{center}
\includegraphics[scale=0.70]{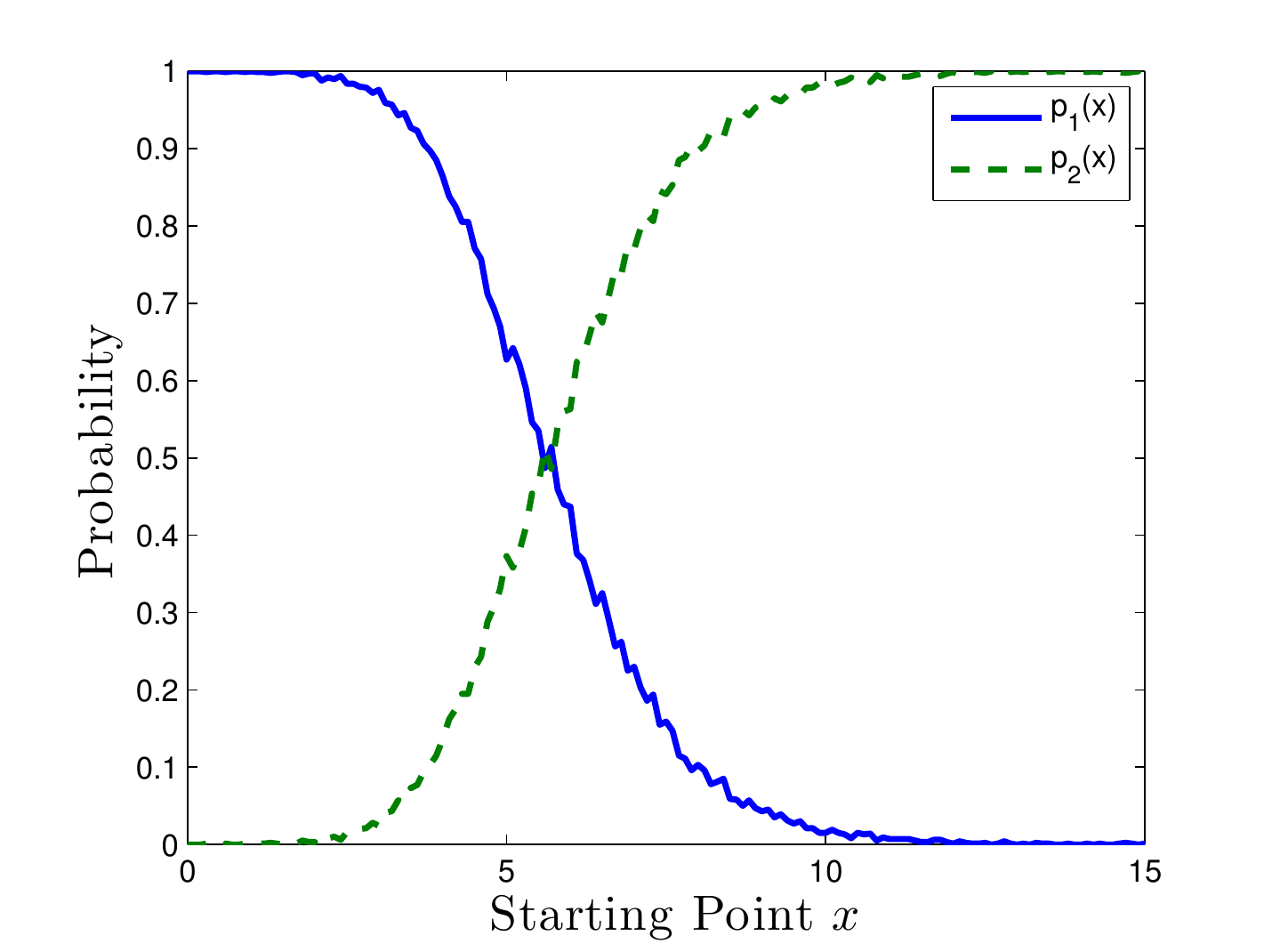}
\caption{Plot of the probabilities that $\frac{N_{t}}{t}$ will converge to $x_{1}<x_{2}$, the stable
fixed point of $x=\lambda(x)=0.9x-\sin(0.6x)+0.5$ as a function of the starting positions.}
\label{p1p2}
\end{center}
\end{figure}

The paper is organized as follows. In Section \ref{MainResultsSection} we will state
all the main results. In particular, we will discuss the law of large numbers in Section \ref{LLNSection}, 
large deviations in Section \ref{LDPSection}, time asymptotics for different regimes in Section \ref{TimeAsympSection}, 
asymptotics for high initial values in Section \ref{LargeInitialSection} 
and marginal and tail probabilities in Section \ref{TailSection}.
The proofs will be given in Section \ref{ProofSection}. Finally some open problems will be suggested in Section \ref{OpenSection}.

\section{Main Results}\label{MainResultsSection}

Throughout the paper, we assume the following conditions hold.
\begin{itemize}
\item
$\lambda:\mathbb{R}_{\geq 0}\rightarrow\mathbb{R}_{\geq 0}$ is an increasing and continuously differentiable function.

\item
$x\mapsto\lambda(x)$ has finitely many fixed points, i.e., the equation $x=\lambda(x)$ has finitely many solutions.
The fixed points are either strictly stable or strictly unstable, i.e., if $x^{\ast}$ is a fixed point, then either $\lambda'(x^{\ast})<1$
or $\lambda'(x^{\ast})>1$.
\end{itemize}

\subsection{Law of Large Numbers}\label{LLNSection}

Assume that $\lambda(z)\leq\beta+\alpha z$ for some $0<\alpha<1$ and $\beta>0$
and $\lambda(x)=x$ has a unique solution $x^{\ast}$. Then,
by Proposition \ref{Expectation}, $\sup_{t>0}\frac{\mathbb{E}[N_{t}]}{t+1}<\infty$. 
Thus $\frac{N_{t}}{t+1}$ is tight. Heuristically, if we have $\frac{N_{t}}{t}\rightarrow x$ a.s. as $t\rightarrow\infty$.
Then, we have as $t\rightarrow\infty$, a.s.,
\begin{equation}
\frac{N_{t}}{t}\rightarrow x,
\qquad
\frac{1}{t}\int_{0}^{t}\lambda\left(\frac{N_{s-}}{s+1}\right)ds\rightarrow\lambda(x).
\end{equation}
Moreover, $M_{t}=N_{t}-\int_{0}^{t}\lambda_{s}ds$ is a martingale and
\begin{equation}
\mathbb{E}[(M_{t})^{4}]\leq C\mathbb{E}[N_{t}^{2}]=O(t^{2})
\end{equation}
by Burkholder-Davis-Gundy inequality and Proposition \ref{Expectation}. Thus, $\frac{M_{t}}{t}\rightarrow 0$ in a.s. by Borel-Cantelli lemma.
Since $\lambda(x)=x$ has a unique solution $x^{\ast}$, we conclude that $\frac{N_{t}}{t}\rightarrow x^{\ast}$ a.s.

\begin{theorem}\label{LLNThm}
Assume that $\lambda(z)$ is increasing, $\alpha$-Lipschitz with $0<\alpha<1$ 
and $x^{\ast}$ is the unique solution to the equation $x=\lambda(x)$
and $x^{\ast}=\infty$ if the solution does not exist, then
\begin{equation}
\frac{N_{t}}{t}\rightarrow x^{\ast},
\end{equation}
in probability as $t\rightarrow\infty$. If we further assume that $\lambda(\cdot)\leq C_{0}<\infty$ for
some universal constant $C_{0}$, then we have the almost sure convergence.
\end{theorem}

\begin{remark}
In Theorem \ref{LLNThm}, if $\lambda(z)\leq\beta+\alpha z$ for some $0<\alpha<1$, $\beta>0$ and $\lambda(z)$
is continuous, then the equation $\lambda(x)=x$ must have at least one solution. In the case that $0$ is the only
solution to $\lambda(x)=x$, we have $\frac{N_{t}}{t}\rightarrow 0$ in probability as $t\rightarrow\infty$.
On the other hand, if $\lambda(z)$ is continuous and the equation $\lambda(x)=x$ has no solution, then
we must have $\lambda(z)>z$ for any $z\geq 0$. Hence, we have $\lambda(z)\geq(1-\epsilon)z+\delta$ for some $\delta,\epsilon>0$
sufficiently small. Note that $(1-\epsilon)z+\delta=z$ if and only if $z=\frac{\delta}{\epsilon}$. Choose $\delta\gg\epsilon$
and by Theorem \ref{LLNThm}, we conclude that $\frac{N_{t}}{t}\rightarrow\infty$ in probability as $t\rightarrow\infty$
if $\lambda(z)$ is continuous and the equation $\lambda(x)=x$ has no solution.
\end{remark}

In Theorem \ref{LLNThm}, we proved that $\frac{N_{t}}{t}$ converges to the unique fixed point of $x=\lambda(x)$
in probability under the $\alpha$-Lipschitz condition for $\lambda(\cdot)$ for some $0<\alpha<1$ and proved
that the convergence is a.s. convergence under a stronger condition.
Next, we compare the underlying stochastic process $Y_{t}:=\frac{N_{t}}{t+1}$ to its deterministic counterpart 
$\bar{Y}_{t}$ where 
$\bar{Y}_{t}$ is the deterministic solution of
\begin{equation}
d\bar{Y}_{t}=\frac{\lambda(\bar{Y}_{t})-\bar{Y}_{t}}{t+1}dt,
\end{equation}
whose asymptotic behavior
is entirely governed by the dynamical system and prove a law of large numbers in the $L^{2}(\mathbb{P})$ norm.
As a by-product, we also get the convergence rate of the underlying stochastic process to its deterministic counterpart.

\begin{theorem}\label{StochDeter}
Assume that $\lambda(\cdot)$ is $\alpha$-Lipschitz for some $0<\alpha<1$, then
$Y_{t}=\frac{N_{t}}{t+1}$ converges to the unique fixed point of $x=\lambda(x)$ as $t\rightarrow\infty$
in $L^{2}(\mathbb{P})$ norm. Moreover, as $t\rightarrow\infty$,
\begin{equation}\label{AsympBound}
\mathbb{E}[(Y_{t}-\bar{Y}_{t})^{2}]=
\begin{cases}
O(\frac{1}{t}) & \text{if $0<\alpha<\frac{1}{2}$}
\\
O(\frac{1}{t^{2(1-\alpha)}}) & \text{if $\frac{1}{2}<\alpha<1$}
\\
O(\frac{\log(t)}{t}) &\text{if $\alpha=\frac{1}{2}$}
\end{cases}.
\end{equation}
\end{theorem}

\begin{theorem}\label{NoFixed}
Assume that $x\mapsto\lambda(x)$ is continuous and increasing.
For any interval $I=[a,b]$ not containing any fixed point
of the equation $x=\lambda(x)$, we have
\begin{equation}
\lim_{t\rightarrow\infty}\mathbb{P}\left(\frac{N_{t}}{t}\in I\right)=0.
\end{equation}
\end{theorem}

\begin{theorem}\label{StablePositive}
Let $x^*$ be any stable fixed point of $\lambda(x)$: $x^* = \lambda(x^*)$ and $0<\lambda'(x^*)<1$, 
the probability that $\frac{N_t}{t}\to x^*$ is non-zero.
\end{theorem}

\begin{theorem}\label{UnstablePositive}
Let $x^*$ be any unstable fixed point of $\lambda(x)$: $x^* = \lambda(x^*)$ and $\lambda'(x^*)>1$, 
the probability that $\frac{N_t}{t}\to x^*$ is zero.
\end{theorem}

\subsection{Large Deviations}\label{LDPSection}

Before we proceed, 
recall that a sequence $(P_{n})_{n\in\mathbb{N}}$ of probability measures on a topological space $\mathbb{X}$ 
satisfies the large deviation principle with rate function $\mathcal{I}:\mathbb{X}\rightarrow\mathbb{R}$ if $\mathcal{I}$ is non-negative, 
lower semicontinuous and for any measurable set $A$, we have
\begin{equation}
-\inf_{x\in A^{o}}\mathcal{I}(x)\leq\liminf_{n\rightarrow\infty}\frac{1}{n}\log P_{n}(A)
\leq\limsup_{n\rightarrow\infty}\frac{1}{n}\log P_{n}(A)\leq-\inf_{x\in\overline{A}}\mathcal{I}(x).
\end{equation}
Here, $A^{o}$ is the interior of $A$ and $\overline{A}$ is its closure. 
We refer to the books by Dembo and Zeitouni \cite{Dembo} and Varadhan \cite{VaradhanII}
and the survey paper by Varadhan \cite{Varadhan} 
for general background of the theory and the applications
of large deviations.

\begin{theorem}\label{LDPThm}
Assume that $\lambda(\cdot)\leq C_{0}<\infty$
and $\lambda(\cdot)$ is $\gamma$-Lipschitz for some $0<\gamma<\infty$.
$\mathbb{P}(\frac{N_{\cdot T}}{T}\in\cdot)$ satisfies a sample path large deviations
on $D[0,1]$ equipped with uniform topology with the rate function
\begin{equation}
I(f)=\int_{0}^{1}\log\left(\frac{f'(\alpha)}{\lambda(f(\alpha)/\alpha)}\right)f'(\alpha)d\alpha
-\int_{0}^{1}\left[f'(\alpha)-\lambda(f(\alpha)/\alpha)\right]d\alpha.
\end{equation}
\end{theorem}

By contraction principle, we get the following scalar large deviation principle.

\begin{corollary}
$\mathbb{P}(N_{t}/t\in\cdot)$ satisfies a large deviation principle with rate function
\begin{equation}
I(x):=\inf_{f\in\mathcal{AC}_{0}^{+}[0,1],f(1)=x}
\int_{0}^{1}\log\left(\frac{f'(\alpha)}{\lambda(f(\alpha)/\alpha)}\right)f'(\alpha)d\alpha
-\int_{0}^{1}\left[f'(\alpha)-\lambda(f(\alpha)/\alpha)\right]d\alpha.
\end{equation}
\end{corollary}

\begin{remark}
It would be interesting to see if one can relax the assumption $\lambda(\cdot)\leq C_{0}<\infty$.
This may not be easy or even possible. For instance, if $\lambda(z)\geq\alpha z$ for some $\alpha>0$, then
for any $\theta>0$, by \eqref{ExampleFormula}, we have
\begin{equation}
\mathbb{E}[e^{\theta N_{t}}]
\geq
\begin{cases}
\left(\frac{\frac{1}{(t+1)^{\alpha}}}{1-\left(1-\frac{1}{(t+1)^{\alpha}}\right)e^{\theta}}\right)^{\gamma}
&\text{if $t<(1-e^{-\theta})^{-\frac{1}{\alpha}}-1$}
\\
\infty &\text{if $t\geq (1-e^{-\theta})^{-\frac{1}{\alpha}}-1$}
\end{cases}.
\end{equation}
Thus for any $\theta>0$, $\mathbb{E}[e^{\theta N_{t}}]=\infty$ for sufficiently large $t$.
\end{remark}

\begin{remark}
Let us define 
\begin{equation}
L(\alpha,f(\alpha),f'(\alpha))
=\log\left(\frac{f'(\alpha)}{\lambda(f(\alpha)/\alpha)}\right)f'(\alpha)
-\left[f'(\alpha)-\lambda(f(\alpha)/\alpha)\right].
\end{equation}
Thus, we are interested to optimize $\int_{0}^{1}L(\alpha,f(\alpha),f'(\alpha))d\alpha$
subject to the constraints $f(0)=0$ and $f(1)=x$. We can write down the Euler-Lagrange equation
\begin{align}
0&=\frac{\partial L}{\partial f}-\frac{d}{d\alpha}\frac{\partial L}{\partial f'}
\\
&=-\frac{\lambda'(f(\alpha)/\alpha)}{\alpha\lambda(f(\alpha)/\alpha)}
+\frac{\lambda'(f(\alpha)/\alpha)}{\alpha}
-\frac{d}{d\alpha}\left[\log(f'(\alpha))-f'(\alpha)\log\lambda\left(\frac{f(\alpha)}{\alpha}\right)\right].
\nonumber
\end{align}
\end{remark}

\begin{remark}\label{ZeroRemark}
If $x=\lambda(x)$, by letting $f(\alpha)=\alpha x$, $0\leq\alpha\leq 1$, we can easily
check that $I(x)=0$.
\end{remark}

\begin{proposition}\label{ZeroProp}
The converse of Remark \ref{ZeroRemark} is also true. In other words, if $I(x)=0$, then $x$ must be a fixed point of $\lambda(x)$ 
and the minimizer is $f(\alpha) = \alpha x$.
\end{proposition}

\begin{remark}
Let $x^{\ast}$ be any unstable fixed point of $x\mapsto\lambda(x)$ and $C$ be any sufficiently small neighborhood
containing $x^{\ast}$. From Theorem \ref{NoFixed} and Theorem \ref{UnstablePositive}, 
we have $\mathbb{P}(\frac{N_{t}}{t}\in C)\rightarrow 0$ as $t\rightarrow\infty$. On the other hand, from Proposition \ref{ZeroProp}, 
we have $I(x^{\ast})=0$, which implies that $\mathbb{P}(\frac{N_{t}}{t}\in C)$ has subexponential decay in time as $t\rightarrow\infty$.
This is consistent with our simulations results which illustrate that when there is a configuration in which $\frac{N_{t}}{t}$ is in a small neighborhood
of $x^{\ast}$, it takes a very long time for the process to exit the neighborhood.
\end{remark}

\subsection{Time Asymptotics in Different Regimes}\label{TimeAsympSection}

When $\lambda(\cdot)$ is $\alpha$-Lipschitz with $0<\alpha<1$, we know that there is a unique
fixed point to $x=\lambda(x)$ and $\frac{N_{t}}{t}$ converges to this unique fixed point as $t\rightarrow\infty$.
In general, $x=\lambda(x)$ may not have any fixed points. If that's the case, then what should be the correct
scaling for $N_{t}$, $\mathbb{E}[N_{t}]$, $\text{Var}[N_{t}]$ etc.? In this section, we study
the different time asymptotics for different regimes. 

\begin{proposition}\label{Expectation}
Assume $\lambda(z)=\beta+\alpha z$, $\alpha,\beta>0$ and $\alpha\neq 1$.

(i)
The expectation is given by
\begin{equation}\label{ExpFormula}
\mathbb{E}[N_{t}]=(t+1)\frac{\beta}{1-\alpha}\left[1-(t+1)^{\alpha-1}\right],
\end{equation}
and thus for $\alpha<1$
\begin{equation}
\lim_{t\rightarrow\infty}\frac{\mathbb{E}[N_{t}]}{t}=\frac{\beta}{1-\alpha},
\end{equation}
and for $\alpha>1$
\begin{equation}
\lim_{t\rightarrow\infty}\frac{\mathbb{E}[N_{t}]}{t^{\alpha}}=\frac{\beta}{\alpha-1}.
\end{equation}

(ii) The variance is given by
\begin{equation}\label{VarFormula}
\text{Var}[N_{t}]=(t+1)^{2}
\bigg[\frac{\frac{\beta}{(1-2\alpha)(1-\alpha)}}{t+1}+\frac{\frac{\beta}{1-\alpha}}{(t+1)^{2-\alpha}}
+\frac{-\frac{2\beta}{1-2\alpha}}{(t+1)^{2(1-\alpha)}}\bigg].
\end{equation}

For $0<\alpha<\frac{1}{2}$, we have
\begin{equation}
\lim_{t\rightarrow\infty}\frac{\text{Var}[N_{t}]}{t}=\frac{\beta}{(1-2\alpha)(1-\alpha)}.
\end{equation}

For $\alpha>\frac{1}{2}$, we have
\begin{equation}
\lim_{t\rightarrow\infty}\frac{\text{Var}[N_{t}]}{t^{2\alpha}}=\frac{2\beta}{2\alpha-1}.
\end{equation}
\end{proposition}

\begin{theorem}\label{CLTThm}
Assume that $\lambda(z)=\beta+\alpha z$, $\alpha,\beta>0$ and $\alpha<\frac{1}{2}$.
Then, we have the central limit theorem
\begin{equation}\label{CLTEqn}
\frac{N_{t}-\frac{\beta}{1-\alpha}t}{\sqrt{t}}\rightarrow N\left(0,\frac{\beta}{(1-2\alpha)(1-\alpha)}\right),
\end{equation}
in distribution as $t\rightarrow\infty$.
\end{theorem}

\begin{proposition}\label{OneHalf}
For $\lambda(z)=\beta+\frac{1}{2}z$, $\beta>0$,
\begin{equation}
\text{Var}[N_{t}]=(t+1)^{2}\left[\frac{2\beta[\log(t+1)-1]}{t+1}+\frac{2\beta}{(t+1)^{3/2}}\right],
\end{equation}
and
\begin{equation}
\lim_{t\rightarrow\infty}\frac{\text{Var}[N_{t}]}{t\log t}=2\beta.
\end{equation}
\end{proposition}

\begin{proposition}\label{One}
For $\lambda(z)=\beta+z$, $\beta>0$,
\begin{equation}
\mathbb{E}[N_{t}]=\beta(t+1)\log(t+1),
\end{equation}
and
\begin{equation}
\text{Var}[N_{t}]=-\beta(t+1)\log(t+1)+2\beta t(t+1)
\end{equation}
and
\begin{equation}
\lim_{t\rightarrow\infty}\frac{\text{Var}[N_{t}]}{t^{2}}=2\beta.
\end{equation}
\end{proposition}

\begin{corollary}\label{LLNAnother}
Let $\lambda(z)=\beta+\alpha z$, $\beta,\alpha>0$, then

(i) For $\alpha<1$,
$\frac{N_{t}}{t}\rightarrow\frac{\beta}{1-\alpha}$ in probability.

(ii) For $\alpha>1$,
$\frac{N_{t}}{t^{\alpha}}\rightarrow\frac{\beta}{\alpha-1}$ in probability.

(iii) For $\alpha=1$,
$\frac{N_{t}}{t\log t}\rightarrow\beta$ in probability.
\end{corollary}

We can also compute the covariance structure explicitly when $\lambda(\cdot)$ is linear. 

\begin{proposition}\label{CovProp}
Assume $\lambda(z)=\beta+\alpha z$, $\alpha,\beta>0$ and $\alpha\notin\{\frac{1}{2},1\}$, for any $t>s$,
\begin{equation}
\text{Cov}[N_{t},N_{s}]
=(t+1)^{\alpha}\left[\frac{\beta}{(1-2\alpha)(1-\alpha)}(s+1)^{1-\alpha}+\frac{\beta}{1-\alpha}
-\frac{2\beta}{1-2\alpha}(s+1)^{\alpha}\right].
\end{equation}
For $\alpha=\frac{1}{2}$,
\begin{equation}
\text{Cov}[N_{t},N_{s}]
=2\beta[-(t+1)^{\frac{1}{2}}(s+1)^{\frac{1}{2}}+(t+1)^{\frac{1}{2}}+(t+1)^{\frac{1}{2}}(s+1)^{\frac{1}{2}}\log(s+1)].
\end{equation}
\end{proposition}

\begin{proposition}\label{CovOneProp}
Assume $\lambda(z)=\beta+z$, $\beta>0$, for any $t>s$,
\begin{equation}
\text{Cov}[N_{t},N_{s}]
=-\beta(t+1)\log(s+1)+2\beta s(t+1).
\end{equation}
\end{proposition}

We have seen that if $\lambda(z)=\alpha z$, $0<\alpha<1$, we have $\lim_{t\rightarrow\infty}\frac{N_{t}}{t}=0$. 
A natural question to ask is under this regime, what should be the correct scaling for $N_{t}$ as time $t$ goes to $\infty$.
Since $\lambda(0)=0$, we need to start the process at some positive inditial condition $\gamma>0$.

\begin{proposition}\label{GammaLimit}
Let the intensity at time $t$ be
\begin{equation}
\lambda_{t}=\frac{\alpha(N_{t-}+\gamma)}{t+1},\qquad\alpha,\gamma>0.
\end{equation}
Then, we have $\mathbb{E}[N_{t}]=\gamma[(t+1)^{\alpha}-1]$ and
\begin{equation}\label{alphaVar}
\text{Var}[N_{t}]=\gamma[(t+1)^{2\alpha}-(t+1)^{\alpha}].
\end{equation}
In particular,
$\lim_{t\rightarrow\infty}\frac{\mathbb{E}[N_{t}]}{t^{\alpha}}=\gamma$,
and $\lim_{t\rightarrow\infty}\frac{\text{Var}[N_{t}]}{t^{2\alpha}}=\gamma$.
Also, for any $t>s$,
\begin{equation}
\text{Cov}[N_{t},N_{s}]=[(s+1)^{\alpha}-1]\left[\gamma(t+1)^{\alpha}+(\gamma-\gamma^{2})\left[\frac{(t+1)^{\alpha}}{(s+1)^{\alpha}}-1\right]\right].
\end{equation}
Moreover, as $t\rightarrow\infty$,
\begin{equation}
\frac{N_{t}}{t^{\alpha}}\rightarrow\chi(\gamma),
\end{equation}
a.s. and in $L^{2}(\mathbb{P})$ where $\chi(\gamma)$ is a random variable with gamma distribution
with parameters $\gamma$ (shape) and $1$ (scale).
\end{proposition}

We end this section with a criterion on whether the point process $N_{t}$ can be explosive or not. 
Essentially, when $\lambda(\cdot)$ is super linear, it gives the explosive regime.

\begin{proposition}\label{ExplosionProp}
Assume that $\int_{0}^{\infty}\frac{1}{\lambda(z)}dz<\infty$. Then, the point process is explosive.
More precisely, $0<\mathbb{P}(\tau<\infty)<1$, where $\tau:=\inf\{t>0:N_{t}=\infty\}$.
\end{proposition}

\subsection{High Initial Value}\label{LargeInitialSection}

One can also study the asymptotics for high initial value $\gamma\rightarrow\infty$. 
In the classical birth-death process, that corresponds to high initial population size. 
The asymptotics results for high initial values can be interesting and useful. For example, 
they are useful in the models of cancer dyanmics, see e.g. Foo and Leder \cite{FooI}, Foo et al. \cite{FooII}.

\begin{proposition}\label{LargeInitial}
Assume that $\lim_{z\rightarrow\infty}\frac{\lambda(z)}{z}=\alpha$. Then, 
\begin{equation}
\sup_{0\leq s\leq t}\left|\frac{N_{s}}{\gamma}-[(s+1)^{\alpha}-1]\right|\rightarrow 0,
\end{equation}
in probability as $\gamma\rightarrow\infty$.
\end{proposition}

We can also study the case when $\lambda(\cdot)$ is sublinear.

\begin{proposition}\label{SublinearInitial}
Assume that $\lim_{z\rightarrow\infty}\frac{\lambda(z)}{z^{\beta}}=\alpha$, where $\alpha>0$ and $0<\beta<1$. Then,
\begin{equation}
\sup_{0\leq s\leq t}\left|\frac{N_{s}}{\gamma^{\beta}}-\frac{\alpha}{1-\beta}[(s+1)^{1-\beta}-1]\right|\rightarrow 0,
\end{equation}
in probability as $\gamma\rightarrow\infty$.
\end{proposition}

\subsection{Marginal and Tail Probabilities}\label{TailSection}

In this section, we are interested to study the marginal probabilities $\mathbb{P}(N_{t}=k)$
for a given $k\in\mathbb{N}\cup\{0\}$ and the asymptotics for the tail probabilities $\mathbb{P}(N_{t}\geq\ell)$
for large $\ell$.
We assume that the initial condition is given by $\lambda_{0}=\lambda(\gamma)$, 
where $\gamma\in\mathbb{R}^{+}$.

\begin{theorem}\label{MarginalProp}
For any $k\in\mathbb{N}\cup\{0\}$,

(i)
\begin{align}\label{MasterFormula}
\mathbb{P}(N_{t}=k)
&=\idotsint_{0<t_{1}<t_{2}<\cdots<t_{k}<t}
\prod_{j=1}^{k}\lambda\left(\frac{\gamma+j-1}{t_{j}+1}\right)
\\
&\qquad\qquad
\cdot
e^{-\int_{0}^{t_{1}}\lambda(\frac{\gamma}{s+1})ds
-\int_{t_{1}}^{t_{2}}\lambda(\frac{\gamma+1}{s+1})ds
-\cdots-\int_{t_{k}}^{t}\lambda(\frac{\gamma+k}{s+1})ds}
dt_{1}dt_{2}\cdots dt_{k}.
\nonumber
\end{align}

(ii) In particular, the void probability is given by
\begin{equation}\label{VoidFormula}
\mathbb{P}(N_{t}=0)=e^{-\int_{0}^{t}\lambda(\frac{\gamma}{s+1})ds}.
\end{equation}

(iii) When $\lambda(z)=\alpha z$, $N_{t}$ follows a negative binomial distribution,
\begin{equation}\label{ExampleFormula}
\mathbb{P}(N_{t}=k)=\binom{k+\gamma-1}{k}\left(1-\frac{1}{(t+1)^{\alpha}}\right)^{k}
\left(\frac{1}{(t+1)^{\alpha}}\right)^{\gamma}.
\end{equation}
For any $t>s>0$ and $k,m\in\mathbb{N}\cup\{0\}$, we have the conditional probability
\begin{equation}\label{ConditionalFormula}
\mathbb{P}(N_{t}=k+m|N_{s}=m)=\binom{k+m+\gamma-1}{k}\left(1-\left(\frac{s+1}{t+1}\right)^{\alpha}\right)^{k}
\left(\left(\frac{s+1}{t+1}\right)^{\alpha}\right)^{\gamma+m}.
\end{equation}
\end{theorem}

For a standard Poisson process $N_{t}$ with constant intensity $\lambda$, the tail probability $\mathbb{P}(N_{t}\geq\ell)$
has the asymptotics $\lim_{\ell\rightarrow\infty}\frac{1}{\ell\log\ell}\log\mathbb{P}(N_{t}\geq\ell)=-1$.
What is the asymptotics for the tail probabilities in our model? In the next result, we will show 
that if $\lambda(\cdot)$ is asymptotically linear, then unlike the standard Poisson process, we have exponential tails
for $\mathbb{P}(N_{t}\geq\ell)$ as $\ell$ goes to infinity.

\begin{theorem}\label{LinearTail}
Assume that $\lim_{z\rightarrow\infty}\frac{\lambda(z)}{z}=\alpha\in(0,\infty)$. Then, for any fixed $t>0$,
\begin{equation}\label{TailLimit}
\lim_{\ell\rightarrow\infty}\frac{1}{\ell}\log\mathbb{P}(N_{t}\geq\ell)=\log\left(1-\frac{1}{(t+1)^{\alpha}}\right).
\end{equation}
\end{theorem}

We have already studied the tail probabilities for $\mathbb{P}(N_{t}\geq\ell)$ when $\lambda(\cdot)$
is asymptotically linear in Theorem \ref{LinearTail}. One can also study the case
when $\lim_{z\rightarrow\infty}\frac{\lambda(z)}{z^{\beta}}=\alpha$, for some $\alpha,\beta>0$. 
When $\beta>1$, $\lambda(z)$ grows super-linearly and there is a positive probability of explosion. 
Therefore in this case, $\mathbb{P}(N_{t}\geq\ell)$ does not vanish to zero as $\ell\rightarrow\infty$.
When $\beta<1$, $\lambda(z)$ grows sub-linearly and $\mathbb{P}(N_{t}\geq\ell)$ does vanish
to zero as $\ell\rightarrow\infty$. The asymptotics of the tail probabilities are studied as follows.

\begin{theorem}\label{SublinearTail}
Assume that $\lim_{z\rightarrow\infty}\frac{\lambda(z)}{z^{\beta}}=\alpha$, for some $\alpha,\beta>0$ and $\beta<1$. Then,
\begin{equation}
\lim_{\ell\rightarrow\infty}\frac{1}{\ell\log\ell}\log\mathbb{P}(N_{t}\geq\ell)=-(1-\beta).
\end{equation}
\end{theorem}

\begin{remark}
For any $\lambda(z)$ that grows slower than any polynomial growth, the tail is the same as the Poisson tail
from Theorem \ref{SublinearTail}. 
For example, for $\lambda(z)$ uniformly bounded, $\lambda(z)=[\log(1+z)]^{\beta}$, $\beta>0$, they
all give the Poisson tail $\lim_{\ell\rightarrow\infty}\frac{1}{\ell\log\ell}\log\mathbb{P}(N_{t}\geq\ell)=-1$.
\end{remark}

\section{Proofs}\label{ProofSection}

\subsection{Proofs of Results in Section \ref{LLNSection}}

\begin{proof}[Proof of Theorem \ref{LLNThm}]
Let us use Poisson embedding. Let $N^{(0)}$ be the Poisson process with intensity $\lambda(0)$.
Conditional on $N^{(0)}$, let $N^{(1)}$ be the inhomogeneous Poisson process with intensity
\begin{equation}
\lambda\left(\frac{N^{(0)}_{t-}}{t+1}\right)-\lambda(0),
\end{equation}
at time $t$. Inductively, conditional on $N^{(0)},N^{(1)},\ldots,N^{(k)}$, $N^{(k+1)}$ is an inhomogeneous Poisson process
with intensity
\begin{equation}
\lambda\left(\frac{N^{(0)}_{t-}+N^{(1)}_{t-}+\cdots N^{(k)}_{t-}}{t+1}\right)-\lambda\left(\frac{N^{(0)}_{t-}+N^{(1)}_{t-}+\cdots N^{(k-1)}_{t-}}{t+1}\right),
\end{equation}
at time $t$. Therefore, we can compute that $\mathbb{E}[N_{t}^{(0)}]=\lambda(0)t$,
\begin{align}
\frac{\mathbb{E}[N_{t}^{(1)}]}{t}&=\frac{1}{t}\mathbb{E}\left[\int_{0}^{t}\lambda\left(\frac{N^{(0)}_{s-}}{s+1}\right)-\lambda(0)ds\right]
\\
&\leq\alpha\frac{1}{t}\int_{0}^{t}\frac{\mathbb{E}[N^{(0)}_{s-}]}{s+1}ds
\nonumber
\\
&\leq\alpha\lambda(0).
\nonumber
\end{align}
and inductively, 
\begin{equation}
\frac{1}{t}\mathbb{E}[N_{t}^{(k)}]\leq\alpha^{k}\lambda(0),\qquad k\in\mathbb{N}.
\end{equation}
Hence, $\mathbb{E}[\sum_{k=0}^{\infty}N^{(k)}_{t}]\leq\sum_{k=0}^{\infty}\alpha^{k}\lambda(0)t<\infty$ 
since $0<\alpha<1$
and $N_{t}=\sum_{k=0}^{\infty}N^{(k)}_{t}$ is well defined a.s. 
Moreover, the compensator of $N_{t}$ is 
\begin{equation}
\sum_{k=0}^{\infty}\int_{0}^{t}\left[\lambda\left(\frac{\sum_{j=0}^{k}N^{(j)}_{s-}}{s+1}\right)
-\lambda\left(\frac{\sum_{j=0}^{k-1}N^{(j)}_{s-}}{s+1}\right)\right]ds=\int_{0}^{t}\lambda\left(\frac{N_{t-}}{t+1}\right)ds
\end{equation}
and hence $N_{t}$ is the self-exciting process we are interested to study. 
By the law of large numbers for Poisson processes,
\begin{equation}
\frac{N^{(0)}_{t}}{t}\rightarrow\lambda(0),
\end{equation}
a.s. as $t\rightarrow\infty$. 
We use induction and assume that $\frac{N^{(j)}_{t}}{t}\rightarrow\lambda^{(j+1)}(0)-\lambda^{(j)}(0)$
a.s. as $t\rightarrow\infty$ for any $j=0,1,\ldots,k-1$. 
Then, the compensator $\Lambda^{(k)}_{t}$ of $N^{(k)}_{t}$ satisfies
\begin{align}
&\frac{1}{t}\int_{0}^{t}\lambda\left(\frac{N^{(0)}_{s-}+N^{(1)}_{s-}+\cdots N^{(k-1)}_{s-}}{s+1}\right)
-\lambda\left(\frac{N^{(0)}_{s-}+N^{(1)}_{s-}+\cdots N^{(k-2)}_{s-}}{s+1}\right)ds
\\
&\qquad\qquad\qquad
\rightarrow\lambda^{(k+1)}(0)-\lambda^{(k)}(0),
\nonumber
\end{align}
a.s. as $t\rightarrow\infty$.
On the other hand, $M^{(k)}_{t}:=N^{(k)}_{t}-\Lambda^{(k)}_{t}$ is a martingale and
\begin{equation}
\frac{M^{(k)}_{t}}{t}\rightarrow 0,
\end{equation}
a.s. as $t\rightarrow\infty$. Hence, we proved that
\begin{equation}
\frac{1}{t}\sum_{j=0}^{k}N^{(j)}_{t}\rightarrow\lambda^{(k+1)}(0),
\end{equation}
a.s. as $t\rightarrow\infty$. As $k\rightarrow\infty$, $\lambda^{(k+1)}(0)\rightarrow x^{\ast}$. 
For any $\epsilon>0$, there exists $K\in\mathbb{N}$ so that for any $k\geq K$,
$|\lambda^{(k+1)}(0)-x^{\ast}|<\frac{\epsilon}{4}$. Thus,
\begin{align}
&\limsup_{t\rightarrow\infty}\mathbb{P}\left(\left|\sum_{k=0}^{\infty}\frac{N^{(k)}_{t}}{t}-x^{\ast}\right|\geq\epsilon\right)
\\
&\leq\limsup_{t\rightarrow\infty}\mathbb{P}
\left(\left|\sum_{k=0}^{K}\frac{N^{(k)}_{t}}{t}-\lambda^{(K+1)}(0)\right|\geq\frac{\epsilon}{2}\right)
\nonumber
\\
&\qquad\qquad\qquad\qquad
+\limsup_{t\rightarrow\infty}\mathbb{P}\left(\left|\sum_{k=K+1}^{\infty}\frac{N^{(k)}_{t}}{t}
-(x^{\ast}-\lambda^{(k+1)}(0))\right|\geq\frac{\epsilon}{2}\right)
\nonumber
\\
&\leq\limsup_{t\rightarrow\infty}\mathbb{P}\left(\sum_{k=K+1}^{\infty}\frac{N^{(k)}_{t}}{t}
\geq\frac{\epsilon}{4}\right)
\nonumber
\\
&\leq\frac{\lambda(0)\sum_{k=K+1}^{\infty}\alpha^{k}}{(\epsilon/4)}.
\nonumber
\end{align}
Since it holds for any $K\in\mathbb{N}$, we get the desired result by letting $K\rightarrow\infty$.

Finally, if we further assume that $\lambda(\cdot)\leq C_{0}<\infty$ for some
universal constant $c_{0}$ and $C_{0}$, then by the large devations results 
in Theorem \ref{LDPThm} and Proposition \ref{ZeroProp},
we get the almost sure convergence.
\end{proof}

\begin{proof}[Proof of Theorem \ref{StochDeter}]
Let us define $Y_{t}=\frac{N_{t}}{t+1}$. It is easy to check that
\begin{equation}
dY_{t}=\frac{\lambda(Y_{t-})-Y_{t}}{t+1}dt+\frac{dM_{t}}{t+1},
\end{equation}
where $M_{t}=N_{t}-\int_{0}^{t}\lambda_{s}ds$ is a martingale. Let us define $\bar{Y}_{t}$ as the deterministic solution of
\begin{equation}
d\bar{Y}_{t}=\frac{\lambda(\bar{Y}_{t})-\bar{Y}_{t}}{t+1}dt.
\end{equation}
We assume that $\bar{Y}_{0}=Y_{0}$. We have $Y_{t}^{2}=\frac{N_{t}^{2}}{(t+1)^{2}}$. Applying It\^{o}'s formula for jump processes,
\begin{align}
\label{eq:Y_t^2 SDE}
&dY_{t}=-\frac{Y_{t}}{t+1}dt+\frac{dN_{t}}{t+1},
\\
&d(Y_{t}^{2})=-\frac{2N_{t}^{2}}{(t+1)^{3}}dt
+\frac{(2N_{t-}+1)dN_{t}}{(t+1)^{2}}
=\frac{-2Y_{t}^{2}}{t+1}dt+\left(\frac{2Y_{t-}}{t+1}+\frac{1}{(t+1)^{2}}\right)dN_{t},
\nonumber
\\
&d\bar{Y}_{t}=\frac{\lambda(\bar{Y}_{t})-\bar{Y}_{t}}{t+1}dt,
\nonumber
\\
&d(\bar{Y}_{t})^{2}=2\bar{Y}_{t}\frac{\lambda(\bar{Y}_{t})-\bar{Y}_{t}}{t+1}dt.
\nonumber
\end{align}
Therefore, we can compute that
\begin{align}
&d(Y_{t}-\bar{Y}_{t})^{2}
=-\frac{2(Y_{t}-\bar{Y}_{t})^{2}}{t+1}dt
+2\frac{\bar{Y}_{t}}{t+1}[\lambda(\bar{Y}_{t})-\lambda(Y_{t})]dt
-2\frac{Y_{t}}{t+1}[\lambda(\bar{Y}_{t})-\lambda(Y_{t})]dt
\\
&\qquad\qquad\qquad\qquad
+\frac{\lambda(Y_{t})}{(t+1)^{2}}dt
+\left[\left(\frac{2Y_{t-}}{t+1}+\frac{1}{(t+1)^{2}}\right)-\frac{2\bar{Y}_{t}}{t+1}\right]dM_{t}.
\nonumber
\end{align}
Define $m(t):=\mathbb{E}[(Y_{t}-\bar{Y}_{t})^{2}]$. Hence,
\begin{align}
\frac{dm(t)}{dt}&=-\frac{2m(t)}{t+1}
+\frac{2}{t+1}\mathbb{E}[(\bar{Y}_{t}-Y_{t})(\lambda(\bar{Y}_{t})-\lambda(Y_{t}))]
+\frac{\mathbb{E}[\lambda(Y_{t})]}{(t+1)^{2}}
\\
&\leq 2(\alpha-1)\frac{m(t)}{t+1}+\frac{\mathbb{E}[\lambda(Y_{t})]}{(t+1)^{2}}.
\nonumber
\end{align}
Since $\lambda(\cdot)$ is $\alpha$-Lipschitz for some $0<\alpha<1$, there exist some $c_{1},c_{2}>0$
and $c_{2}<1$ so that $\lambda(z)\leq c_{1}+c_{2}z$. Under this condition, 
by Proposition \ref{Expectation}, $\frac{\mathbb{E}[N_{t}]}{t+1}\leq K$ uniformly in $t$ for some $K>0$
and $\mathbb{E}[\lambda_{t}]=\mathbb{E}[\lambda(\frac{N_{t-}}{t+1})]\leq c_{1}+c_{2}K$ uniformly in $t$.
Therefore,
\begin{equation}
\frac{dm(t)}{dt}\leq\frac{2(\alpha-1)}{t+1}m(t)+\frac{c_{1}+c_{2}K}{(t+1)^{2}},\qquad m(0)=0.
\end{equation}
It is easy to verify that the solution to the ODE
\begin{equation}
\frac{dm(t)}{dt}=\frac{2(\alpha-1)}{t+1}m(t)+\frac{c_{1}+c_{2}K}{(t+1)^{2}},\qquad m(0)=0,
\end{equation}
when $\alpha\neq\frac{1}{2}$ is given by
\begin{equation}
m(t)=\frac{c_{1}+c_{2}K}{1-2\alpha}\left[\frac{1}{t+1}-\frac{1}{(t+1)^{2(1-\alpha)}}\right].
\end{equation}
When $\alpha=\frac{1}{2}$, the solution is given by
\begin{equation}
m(t)=\frac{(c_{1}+c_{2}K)\log(t+1)}{t+1}.
\end{equation}
Therefore, we proved \eqref{AsympBound}
and it is clear that $m(t)\rightarrow\infty$ as $t\rightarrow\infty$.
Since $\bar{Y}_{t}$ converges to the unique fixed point deterministically, we conclude
that $\frac{N_{t}}{t+1}$ converges to the same value in the $L^{2}(\mathbb{P})$ norm.
\end{proof}

\begin{proof}[Proof of Theorem \ref{NoFixed}]
For any $x$ that is not a fixed point of the equation $x=\lambda(x)$, then either
$x>\lambda(x)$ or $x<\lambda(x)$. Let us assume without loss of generality that $\lambda(x)>x$. By continuity of $\lambda(x)$,
there exists a sufficiently small $\epsilon>0$ such that $\lambda(x-\epsilon)>x+\epsilon$. 
We claim that
\begin{equation}\label{ZeroIdentity}
\mathbb{P}\left(\omega:\exists N(\omega), \forall t\geq N(\omega),x-\epsilon<\frac{N_{t-}}{t+1}<x+\epsilon\right)=0.
\end{equation}
Notice that if $x-\epsilon<\frac{N_{t-}}{t+1}<x+\epsilon$ for any $t\geq N(\omega)$, then, from
the monotonicity of the function $\lambda(x)$, we have
\begin{equation}
\lambda_{t}=\lambda\left(\frac{N_{t-}}{t+1}\right)\in[\lambda(x-\epsilon),\lambda(x+\epsilon)],
\end{equation}
which is bounded below by $\lambda(x-\epsilon)$. But for a standard Poisson process $N_{t}$ with constant intensity
$\lambda(x-\epsilon)$, $\frac{N_{t-}}{t+1}\rightarrow\lambda(x-\epsilon)>x+\epsilon$ almost surely,
which implies \eqref{ZeroIdentity}. And \eqref{ZeroIdentity} implies that
\begin{equation}
\lim_{t\rightarrow\infty}\mathbb{P}\left(\frac{N_{t}}{t}\in(x-\epsilon,x+\epsilon)\right)
=\lim_{t\rightarrow\infty}
\mathbb{P}\left(\frac{N_{t-}}{t+1}\in(x-\epsilon,x+\epsilon)\right)=0.
\end{equation}
Note that for the $\epsilon$ above, it depends on $x$. Now, consider any $I=[a,b]$ not containing
any fixed point of $x=\lambda(x)$ and assume $\lambda(x)>x$ for any $x\in[a,b]$. 
Since $x\mapsto\lambda(x)$ is continuous, there exists $\epsilon>0$
sufficiently small so that uniformly in $x\in[a,b]$, $\lambda(x-\epsilon)>x+\epsilon$.
Hence the proof is complete.
\end{proof}

\begin{proof}[Proof of Theorem \ref{StablePositive}]
If $x^*$ is the unique and stable fixed point of $\lambda(x)$, then $\frac{N_t}{t}\to x^*$ as $t\to\infty$ by using the previous result. Therefore, 
it is sufficient to show the following lemma.
\begin{lemma}
\label{lma: Poisson boundedness}
Given that $x^*$ a stable fixed point of $\lambda(x)$, there exists $\epsilon>0$ and $t_0>0$ such that conditional on 
$\frac{N_{t_0}}{t_0} \in (x^*-\epsilon, x^*+\epsilon)$,
\begin{equation}
	\mathbb{P}\left(\frac{N_{t}}{t} \in (x^*-2\epsilon, x^*+2\epsilon), \forall t\geq t_0 \right)>0
\end{equation}
\end{lemma}
If Lemma \ref{lma: Poisson boundedness} holds, as long as $\frac{N_{t}}{t} \in (x^*-2\epsilon, x^*+2\epsilon)$, $\forall t\geq t_0$, we can modify 
$\lambda(x)$ outside $(x^*-2\epsilon, x^*+2\epsilon)$ so that $x^*$ is the unique fixed point. In addition, taking into account that 
$\frac{N_t}{t}$ is Markov and the event that $\frac{N_{t_0}}{t_0} \in (x^*-\epsilon, x^*+\epsilon)$ for some $t_0>0$ has positive probability,
the proof is completed.
\end{proof}

\begin{proof}[Proof of Lemma \ref{lma: Poisson boundedness}]
Because $x^{\ast}$ is a stable fixed point, $|\lambda'(x^{\ast})|<1$, 
and we can find $\epsilon>0$ and $\delta>0$, such that $x^* - \epsilon < \lambda(x) < x^* + \epsilon$ for 
$x\in (x^*-\epsilon-\delta, x^*+\epsilon+\delta)$. 

Define a stopping time $\tau = \inf\{t\geq t_0: \frac{N_{t}}{t} \notin (x^*-\epsilon-\delta, x^*+\epsilon+\delta)\}$.
By using the coupling argument, we can construct two Poisson processes $N^1_t$ and $N^2_t$ with the intensity $\lambda_1=x^* - \epsilon$ 
and $\lambda_2=x^* + \epsilon$, respectively, such that 
\begin{equation}
	N^1_t \leq N_t \leq N^2_t,\quad  t\leq \tau
\end{equation}
almost surely. Therefore, if $\tau_1 = \inf\{t\geq t_0: \frac{N^1_{t}}{t} \geq x^*+\epsilon+\delta\}$ and 
$\tau_2 = \inf\{t\geq t_0: \frac{N^2_{t}}{t} \leq x^*-\epsilon-\delta\}$, we have $\tau \geq \tau_1 \wedge \tau_2$ almost surely and 
\begin{equation*}
	\mathbb{P}(t_0 \leq \tau < \infty) \leq \mathbb{P}(t_0 \leq  \tau_1 \wedge \tau_2 < \infty) 
	\leq \mathbb{P}(t_0 \leq \tau_1 < \infty) + \mathbb{P}(t_0 \leq \tau_2 < \infty)
\end{equation*}
By the strong law of the large numbers $\mathbb{P}(t_0 \leq \tau_1 < \infty)$, $\mathbb{P}(t_0 \leq \tau_2 < \infty)\to 0$ as $t_0\to\infty$.
Finally, by letting $\delta<\epsilon$,
\begin{align*}
	&\mathbb{P}\left(\frac{N_{t}}{t} \in (x^*-2\epsilon, x^*+2\epsilon), \forall t\geq t_0 \right)
	\geq 1 - \mathbb{P}(t_0\leq \tau < \infty)\\
	&\quad \geq 1 - \mathbb{P}(t_0 \leq \tau_1 < \infty) - \mathbb{P}(t_0 \leq \tau_2 < \infty) > 0
\end{align*}
for sufficiently large $t_0$ and we complete the proof. 
\end{proof}

\begin{proof}[Proof of Theorem \ref{UnstablePositive}]
Let $x^{\ast}$ be a strictly unstable fixed point. There exists a sufficiently small neighborhood $C$ containing $x^{\ast}$
so that for any $x\in C$ and $x\geq x^{\ast}$, we have $\lambda(x)\geq x$ and for any $x\in C$ and $x\leq x^{\ast}$, we have $\lambda(x)\leq x$.
Let $Y_{t}=\frac{N_{t}}{t+1}$, where $N_{t}$ is the simple point process with intensity $\lambda(\frac{N_{t-}+\gamma}{t+1})$ at time $t$
and let $\tilde{Y}_{t}=\frac{\tilde{N}_{t}}{t+1}$, where $\tilde{N}_{t}$ is the simple point process
with intensity $\frac{\tilde{N}_{t-}+\gamma}{t+1}$ at time $t$.
Finally, we introduce the process $\hat{Y}_{t}=\frac{\hat{N}_{t}+\gamma}{t+1}$ so that the intensity of $\hat{N}_{t}$
is $\lambda(\frac{\hat{N}_{t-}+\gamma}{t+1})$ when $\hat{Y}_{t}\notin C$ and the intensity of $\hat{N}_{t}$
is $\frac{\hat{N}_{t-}+\gamma}{t+1}$ when $\hat{Y}_{t}\in C$. Since $\lambda(x)\geq x$ and for any $x\in C$ and $x\geq x^{\ast}$, 
and $\lambda(x)\leq x$ for any $x\in C$ and $x\leq x^{\ast}$, 
it is clear that $\mathbb{P}(\lim_{t\rightarrow\infty}Y_{t}=x^{\ast})\leq\mathbb{P}(\lim_{t\rightarrow\infty}\hat{Y}_{t}=x^{\ast})$.
On the other hand, for the process $\tilde{Y}_{t}$, we proved in Proposition \ref{GammaLimit} that
$\mathbb{P}(\lim_{t\rightarrow\infty}\tilde{Y}_{t}=\chi(\gamma))=1$, where $\chi(\gamma)$ is a gamma random variable with shape $\gamma$ and scale $1$.
Therefore, for any $x\in\mathbb{R}^{+}$ and hence $x^{\ast}$, $\mathbb{P}(\lim_{t\rightarrow\infty}\tilde{Y}_{t}=x)=0$.
Since $\hat{Y}_{t}$ shares the same dynamics as $\tilde{Y}_{t}$ in $C$, 
we have $\mathbb{P}(\lim_{t\rightarrow\infty}\hat{Y}_{t}=x^{\ast})=0$, which implies that $\mathbb{P}(\lim_{t\rightarrow\infty}Y_{t}=x^{\ast})=0$.
\end{proof}

\subsection{Proofs of Results in Section \ref{LDPSection}}

\begin{proof}[Proof of Theorem \ref{LDPThm}]
To prove the lower bound, it suffices to prove that (since we have the superexponential estimates \eqref{super1} and \eqref{super2})
\begin{equation}
\liminf_{T\rightarrow\infty}\frac{1}{T}\log\mathbb{P}\left(\frac{N_{\alpha_{1}T}}{T}\in B_{\epsilon}(x_{1}),
\ldots,\frac{N_{\alpha_{n}T}}{T}\in B_{\epsilon}(x_{n})\right)\geq -I(f),
\end{equation}
where $B_{\epsilon}(x_{i})$ are open balls centered at $x_{i}$ with radius $\epsilon>0$ and $f(\alpha)$, $0\leq\alpha\leq 1$ 
is piecewise linear such that $f(\alpha_{j})=x_{j}$ for any $j$, where $0=\alpha_{0}<\alpha_{1}<\cdots<\alpha_{n}=1$.

We tilt $\lambda_{s}$ to $\frac{x_{j}-x_{j-1}}{\alpha_{j}-\alpha_{j-1}}$ for $\alpha_{j-1}T<s\leq\alpha_{j}T$. 
Under the new measure, let us use induction. Assume that $\frac{N_{\alpha_{j-1}T}}{T}\rightarrow x_{j-1}$.
Then, if we do not tilt on $[\alpha_{j-1}t,\alpha_{j}t]$, then, we get
\begin{equation}
\frac{N_{\alpha_{j}T}}{T}=\frac{N_{\alpha_{j-1}T}}{T}+\frac{N[\alpha_{j-1}T,\alpha_{j}T]}{T}\rightarrow x_{j-1}
+\frac{x_{j}-x_{j-1}}{\alpha_{j}-\alpha_{j-1}}(\alpha_{j}-\alpha_{j-1})=x_{j}.
\end{equation}

Let $\hat{\mathbb{P}}$ denote the tilted probability measure and
\begin{equation}
A_{T}:=\left\{\frac{N_{\alpha_{1}T}}{T}\in B_{\epsilon}(x_{1}),
\ldots,\frac{N_{\alpha_{n}T}}{T}\in B_{\epsilon}(x_{n})\right\}.
\end{equation}
The tilted probability measure $\hat{\mathbb{P}}$ is absolutely continuous w.r.t. $\mathbb{P}$ and we
have the following Girsanov formula, (For the theory of absolute continuity
for point processes and its Girsanov formula, we refer to Lipster and Shiryaev \cite{Lipster}.)
\begin{equation}
\frac{d\hat{\mathbb{P}}}{d\mathbb{P}}\bigg|_{\mathcal{F}_{T}}
=\exp\left\{\sum_{j=1}^{n}\int_{\alpha_{j-1}T}^{\alpha_{j}T}
\log\left(\frac{\frac{x_{j}-x_{j-1}}{\alpha_{j}-\alpha_{j-1}}}{\lambda_{s}}\right)dN_{s}+\int_{\alpha_{j-1}T}^{\alpha_{j}T}
\lambda_{s}-\left(\frac{x_{j}-x_{j-1}}{\alpha_{j}-\alpha_{j-1}}\right)ds\right\}.
\end{equation}
By Jensen's inequality, we have
\begin{align}
\frac{1}{T}\log\mathbb{P}\left(A_{T}\right)
&=\frac{1}{T}\log\int_{A_{T}}\frac{d\mathbb{P}}{d\hat{\mathbb{P}}}d\hat{\mathbb{P}}
\\
&=\frac{1}{T}\log\hat{\mathbb{P}}\left(A_{T}\right)
+\frac{1}{T}\log\left[\frac{1}{\hat{\mathbb{P}}\left(A_{T}\right)}
\int_{A_{T}}\frac{d\mathbb{P}}{d\hat{\mathbb{P}}}d\hat{\mathbb{P}}\right]\nonumber
\\
&\geq\frac{1}{T}\log\hat{\mathbb{P}}\left(A_{T}\right)-\frac{1}{\hat{\mathbb{P}}\left(A_{T}\right)}
\cdot\frac{1}{T}\hat{\mathbb{E}}\left[1_{A_{T}}\log\frac{d\hat{\mathbb{P}}}{d\mathbb{P}}\right].\nonumber
\end{align}
Hence, we have
\begin{align}
&\liminf_{T\rightarrow\infty}\frac{1}{T}\log\mathbb{P}\left(\frac{N_{\alpha_{1}T}}{T}\in B_{\epsilon}(x_{1}),
\ldots,\frac{N_{\alpha_{n}T}}{T}\in B_{\epsilon}(x_{n})\right)
\\
&\geq-\lim_{T\rightarrow\infty}\frac{1}{T}\hat{\mathbb{E}}\left[\sum_{j=1}^{n}\int_{\alpha_{j-1}T}^{\alpha_{j}T}
\log\left(\frac{\frac{x_{j}-x_{j-1}}{\alpha_{j}-\alpha_{j-1}}}{\lambda_{s}}\right)dN_{s}+\int_{\alpha_{j-1}T}^{\alpha_{j}T}
\lambda_{s}-\left(\frac{x_{j}-x_{j-1}}{\alpha_{j}-\alpha_{j-1}}\right)ds\right]\nonumber
\\
&=-\sum_{j=1}^{n}
\log\left(\frac{\frac{x_{j}-x_{j-1}}{\alpha_{j}-\alpha_{j-1}}}{\lambda(x_{j}/\alpha_{j})}\right)(x_{j}-x_{j-1})+
(\alpha_{j}-\alpha_{j-1})\lambda(x_{j}/\alpha_{j})-(x_{j}-x_{j-1})
\nonumber
\\
&=-I(f).\nonumber
\end{align}
To prove the upper bound for compact sets, it is sufficient to prove that for any piecewise linear $f\in\mathcal{AC}_{0}[0,1]$,
\begin{equation}
\limsup_{\epsilon\rightarrow 0}\limsup_{t\rightarrow\infty}
\frac{1}{t}\log\mathbb{P}\left(\frac{N_{\cdot t}}{t}\in B_{\epsilon}(f)\right)\leq-I(f).
\end{equation}
To prove the upper bound for closed sets instead of compact sets, one needs to prove some superexponential estimates which will
be discussed later.

Notice that
\begin{equation}
1=\mathbb{E}\left[e^{\int_{0}^{1}\Phi(\alpha)dN_{\alpha T}
-\int_{0}^{1}(e^{\Phi(\alpha)}-1)\lambda(\frac{N_{\alpha T}}{1+\alpha T})d(\alpha T)}\right],
\end{equation}
for any bounded function $\Phi$. That is because for any $f(s,\omega)$ which is bounded, progressively measurable
and $\mathcal{F}_{t}$-predictable, 
\begin{equation}
\exp\left\{\int_{0}^{t}f(s,\omega)dN_{s}-\int_{0}^{t}(e^{f(s,\omega)}-1)\lambda_{s}ds\right\},
\end{equation}
is a martingale.

Let us choose the test functions $\Phi$ as a step function and assume that there exists a sequence
$0=\alpha_{0}<\alpha_{1}<\cdots<\alpha_{M}=1$ such that 
$\Phi(\alpha)=\beta_{j}$ for any $\alpha_{j-1}<\alpha<\alpha_{j}$, $1\leq j\leq M$.

For $\frac{N_{\cdot T}}{T}\in B_{\epsilon}(f)$, we have
\begin{align}
&\left|\int_{\alpha_{j-1}}^{\alpha_{j}}\Phi(\alpha)d\left(\frac{dN_{\alpha T}}{T}\right)
-\int_{\alpha_{j-1}}^{\alpha_{j}}\Phi(\alpha)f'(\alpha)d\alpha\right|
\\
&=\left|\beta_{j}\frac{N_{\alpha_{j}T}-N_{\alpha_{j-1}T}}{T}
-\beta_{j}(f(\alpha_{j})-f(\alpha_{j-1}))\right|
\leq 2|\beta_{j}|\epsilon.\nonumber
\end{align}
Moreover,
\begin{equation}
\left|\int_{0}^{1}(e^{\Phi}-1)\lambda\left(\frac{N_{\alpha T}}{1+\alpha T}\right)d\alpha
-\int_{0}^{1}(e^{\Phi}-1)\lambda\left(\frac{f(\alpha)}{\alpha}\right)d\alpha\right|
\leq\sup_{0\leq\alpha\leq 1}|e^{\Phi}-1|\gamma\cdot\epsilon.
\end{equation}
Hence, by Chebychev's inequality, we have
\begin{align}
&\limsup_{T\rightarrow\infty}
\frac{1}{T}\log\mathbb{P}\left(\frac{N_{\alpha T}}{T}\in B_{\epsilon}(f), 0\leq\alpha\leq 1\right)
\\
&\leq-\left\{\int_{0}^{1}\Phi(\alpha)f'(\alpha)d\alpha
-\int_{0}^{1}(e^{\Phi}-1)\lambda\left(\frac{f(\alpha)}{\alpha}\right)d\alpha\right\}
\nonumber
\\
&\qquad\qquad\qquad\qquad
+2\epsilon\sum_{j=1}^{M}|\beta_{j}|+\sup_{0\leq\alpha\leq 1}|e^{\Phi}-1|\gamma\cdot\epsilon.
\nonumber
\end{align}
Hence, we have
\begin{align}
&\limsup_{\epsilon\rightarrow 0}\limsup_{T\rightarrow\infty}
\frac{1}{T}\log\mathbb{P}\left(\frac{N_{\cdot T}}{T}\in B_{\epsilon}(f)\right)
\\
&\leq
-\left\{\int_{0}^{1}\Phi(\alpha)f'(\alpha)d\alpha
-\int_{0}^{1}(e^{\Phi}-1)\lambda\left(\frac{f(\alpha)}{\alpha}\right)d\alpha\right\}.
\nonumber
\end{align}
We can optimize over $\Phi$ by choosing $\Phi=\log(\frac{f'(\alpha)}{\lambda(f(\alpha)/\alpha)})$.
Hence, we proved that
\begin{equation}
\limsup_{\epsilon\rightarrow 0}\limsup_{T\rightarrow\infty}
\frac{1}{T}\log\mathbb{P}\left(\frac{N_{\cdot T}}{T}\in B_{\epsilon}(f)\right)\leq-I(f).
\end{equation}

Finally, we need to obtain the superexponential estimates in order to prove the upper bound for closed sets
instead of compact sets in the topology of uniform convergence. This is not difficult because the jump rate $\lambda(\cdot)\leq C_{0}$.
We have the following superexponential estimates,
\begin{align}
&\limsup_{K\rightarrow\infty}\limsup_{T\rightarrow\infty}\frac{1}{T}\log\mathbb{P}\left(\sup_{\alpha\in[0,1]}\frac{N_{\alpha T}}{T}\geq K\right)
\label{super1}
\\
&=\limsup_{K\rightarrow\infty}\limsup_{T\rightarrow\infty}\frac{1}{T}\log\mathbb{P}(N_{T}\geq KT)=-\infty,\nonumber
\end{align}
and for any $\delta>0$
\begin{align}
&\limsup_{\epsilon\rightarrow 0}\limsup_{T\rightarrow\infty}\frac{1}{T}\log\mathbb{P}
\left(\sup_{|\alpha-\beta|\leq\epsilon,0\leq\alpha,\beta\leq 1}\left|\frac{N_{\alpha T}}{T}-\frac{N_{\beta T}}{T}\right|\geq\delta\right)
\label{super2}
\\
&\leq\limsup_{\epsilon\rightarrow 0}\limsup_{T\rightarrow\infty}\frac{1}{T}\log\mathbb{P}
\left(\exists 1\leq j\leq\left[1/\epsilon\right]:N[j\epsilon T,(j+1)\epsilon T]\geq\frac{T\delta}{2}\right)\nonumber
\\
&\leq\limsup_{\epsilon\rightarrow 0}\limsup_{T\rightarrow\infty}\frac{1}{T}\log\left[1/\epsilon\right]\mathbb{P}
\left(N^{+}[0,\epsilon T]>\frac{T\delta}{2}\right),\nonumber
\end{align}
where $N^{+}$ is the Poisson process with constant rate $C_{0}$. Applying Chebychev's inequality and 
setting $\theta=\log\left(\frac{1+\epsilon}{\epsilon}\right)$, we have
\begin{align}
&\limsup_{\epsilon\rightarrow 0}\limsup_{T\rightarrow\infty}\frac{1}{T}\log\left[1/\epsilon\right]\mathbb{P}
\left(N^{+}[0,\epsilon T]>\frac{T\delta}{2}\right)
\\
&\leq\limsup_{\epsilon\rightarrow 0}\limsup_{T\rightarrow\infty}\frac{1}{T}\log
\left[1/\epsilon\right]e^{C_{0}(e^{\theta}-1)\epsilon T-\theta\delta T/2}\nonumber
\\
&\leq\limsup_{\epsilon\rightarrow 0}\left\{C_{0}-\log\left(\frac{1+\epsilon}{\epsilon}\right)\frac{\delta}{2}\right\}=-\infty.\nonumber
\end{align}
Hence, we have, for any closet set $C$,
\begin{equation}
\limsup_{T\rightarrow\infty}\frac{1}{T}\log\mathbb{P}\left(\frac{N_{\cdot T}}{T}\in C\right)\leq-\inf_{f\in C}I(f).
\end{equation}
\end{proof}

\begin{proof}[Proof of Proposition \ref{ZeroProp}]
We first note that $I(x)$ is a good rate function and therefore there must be a function $f\in\mathcal{AC}_0^+[0,1]$ with $f(0)=0$ and $f(1)=x$ 
such that
\begin{equation}
	I(x)= \int_0^1 L(\alpha, f(\alpha), f'(\alpha))d\alpha=0.
\end{equation}
We know that $L(\alpha, f(\alpha), f'(\alpha))\geq0$ and $L(\alpha, f(\alpha), f'(\alpha))=0$ if and only if 
\begin{equation}
	\label{eq:ODE of f}
	f'(\alpha)=\lambda\left(\frac{f(\alpha)}{\alpha}\right).
\end{equation}
As the limit $\alpha\rightarrow0^+$, we have
\begin{equation*}
	f'(0) = \lim_{\alpha\rightarrow 0^+}\lambda\left(\frac{f(\alpha)}{\alpha}\right) 
	= \lim_{\alpha\rightarrow 0^+}\lambda\left(\frac{f(\alpha)-f(0)}{\alpha}\right)
	= \lambda(f'(0)).
\end{equation*}
Thus, $f'(0)$ must be a fixed point $x^*$ of $\lambda(x)$, i.e., $x^*=\lambda(x^*)$. 
Then we discretize (\ref{eq:ODE of f}) by using the Euler method:
\begin{equation}
	\label{eq:Euler method}
	\begin{cases}
		f_1 = f_0 + \Delta\alpha \lambda(f'(0)), &f_0=0\\
		f_{n+1} = f_n + \Delta\alpha \lambda\left(\frac{f_n}{n\Delta \alpha}\right), &1\leq n < N=\frac{1}{\Delta \alpha}.
	\end{cases}
\end{equation}
By using the fact that $f'(0) = x^* = \lambda(x^*)$, it is easy to see that $f_n = x^* n\Delta \alpha$ for all $n$ and 
$f_N=x^*N\Delta \alpha =x^*$. When $\Delta\alpha\rightarrow 0$, $\{f_n\}_{n=0}^{N}$ obtained by (\ref{eq:Euler method}) converges to the solution 
of (\ref{eq:ODE of f}). Therefore, as $\Delta\alpha\to 0$, $x^* = f_N \rightarrow f(1) = x$ so $x=x^*$ which is a fixed point of 
$\lambda(x)$. In addition, $f(\alpha)$ must be linear: $f(\alpha) = \alpha x$, for $\alpha\in [0,1]$. 
\end{proof}

\subsection{Proofs of Results in Section \ref{TimeAsympSection}}

\begin{proof}[Proof of Proposition \ref{Expectation}]
(i) 
Let us recall that
\begin{align}
&dY_{t}=-\frac{Y_{t}}{t+1}dt+\frac{dN_{t}}{t+1},
\\
&d(Y_{t}^{2})
=\frac{-2Y_{t}^{2}}{t+1}dt+\left(\frac{2Y_{t-}}{t+1}+\frac{1}{(t+1)^{2}}\right)dN_{t}.
\nonumber
\end{align}

Let us assume that $\lambda(z)=\alpha z+\beta$, where $\alpha,\beta>0$ and $\alpha<1$
so that there is a unique fixed point to the equation $z=\lambda(z)$ at $z^{\ast}=\frac{\beta}{1-\alpha}$.

Let $m_{1}(t)=\mathbb{E}[Y_{t}]$ and assume $Y_{0}=0$, then,
\begin{equation}
\frac{dm_{1}(t)}{dt}=\frac{(\alpha-1)m_{1}(t)+\beta}{t+1},\qquad m_{1}(0)=0,
\end{equation}
which implies that
\begin{equation}
m_{1}(t)=\frac{\beta}{1-\alpha}\left[1-(t+1)^{\alpha-1}\right],
\end{equation}
which yields \eqref{ExpFormula}.

(ii)
Let $m_{2}(t)=\mathbb{E}[Y_{t}^{2}]$. Then,
\begin{equation}
\frac{dm_{2}(t)}{dt}=\frac{2(\alpha-1)m_{2}(t)}{t+1}+\frac{2\beta m_{1}(t)}{t+1}+\frac{\beta+\alpha m_{1}(t)}{(t+1)^{2}},
\qquad m_{2}(0)=0.
\end{equation}
Therefore,
\begin{align}
&\frac{dm_{2}(t)}{dt}=\frac{2(\alpha-1)m_{2}(t)}{t+1}+\frac{2\beta^{2}}{1-\alpha}\left[\frac{1}{t+1}-\frac{1}{(t+1)^{2-\alpha}}\right]
\\
&\qquad\qquad
+\frac{\frac{\beta}{1-\alpha}}{(t+1)^{2}}
-\frac{\beta\alpha}{1-\alpha}\frac{1}{(t+1)^{3-\alpha}},
\qquad m_{2}(0)=0.
\nonumber
\end{align}
Consider $m_{2}(t)=\frac{C_{1}}{t+1}+\frac{C_{2}}{(t+1)^{2-\alpha}}+C_{3}+\frac{C_{4}}{(t+1)^{1-\alpha}}+\frac{C_{5}}{(t+1)^{2(1-\alpha)}}$. Then,
\begin{align}
&-\frac{C_{1}}{(t+1)^{2}}-(2-\alpha)\frac{C_{2}}{(t+1)^{3-\alpha}}-(1-\alpha)\frac{C_{4}}{(t+1)^{2-\alpha}}+2(\alpha-1)\frac{C_{5}}{(t+1)^{3-2\alpha}}
\\
&=\frac{2(\alpha-1)C_{1}}{(t+1)^{2}}+\frac{2(\alpha-1)C_{2}}{(t+1)^{3-\alpha}}
+\frac{2(\alpha-1)C_{3}}{t+1}
\nonumber
\\
&\qquad\qquad
+\frac{2\beta^{2}}{1-\alpha}\frac{1}{t+1}
+\frac{\frac{\beta}{1-\alpha}}{(t+1)^{2}}
-\frac{\beta\alpha}{1-\alpha}\frac{1}{(t+1)^{3-\alpha}}
\nonumber
\\
&\qquad\qquad\qquad
+2(\alpha-1)\frac{C_{4}}{(t+1)^{2-\alpha}}-\frac{2\beta^{2}}{1-\alpha}\frac{1}{(t+1)^{2-\alpha}}
+2(\alpha-1)\frac{C_{5}}{(t+1)^{3-2\alpha}}.
\nonumber
\end{align}
Therefore, we have
\begin{align}
&C_{1}=\frac{\beta}{(1-2\alpha)(1-\alpha)},
\\
&C_{2}=\frac{\beta}{1-\alpha},
\nonumber
\\
&C_{3}=\frac{\beta^{2}}{(1-\alpha)^{2}},
\nonumber
\\
&C_{4}=\frac{-2\beta^{2}}{(1-\alpha)^{2}}.
\nonumber
\end{align}
Finally, since $m_{2}(0)=C_{1}+C_{2}+C_{3}+C_{4}+C_{5}=0$, we have
\begin{equation}
C_{5}=-C_{1}-C_{2}-C_{3}-C_{4}=-\frac{\beta}{1-\alpha}\left[1-\frac{\beta}{1-\alpha}+\frac{1}{1-2\alpha}\right].
\end{equation}
Therefore,
\begin{align}
\text{Var}[N_{t}]&=(t+1)^{2}\left[m_{2}(t)-m_{1}(t)^{2}\right]
\\
&=(t+1)^{2}
\bigg[\frac{\frac{\beta}{(1-2\alpha)(1-\alpha)}}{t+1}+\frac{\frac{\beta}{1-\alpha}}{(t+1)^{2-\alpha}}+\frac{\beta^{2}}{(1-\alpha)^{2}}
\nonumber
\\
&
\qquad\qquad
+\frac{\frac{-2\beta^{2}}{(1-\alpha)^{2}}}{(t+1)^{1-\alpha}}
+\frac{-\frac{\beta}{1-\alpha}\left[1-\frac{\beta}{1-\alpha}+\frac{1}{1-2\alpha}\right]}{(t+1)^{2(1-\alpha)}}
\nonumber
\\
&\qquad\qquad
-\frac{\beta^{2}}{(1-\alpha)^{2}}-\frac{\beta^{2}}{(1-\alpha)^{2}}\frac{1}{(t+1)^{2(1-\alpha)}}
+2\frac{\beta^{2}}{(1-\alpha)^{2}}\frac{1}{(t+1)^{1-\alpha}}\bigg]
\nonumber
\\
&=(t+1)^{2}
\bigg[\frac{\frac{\beta}{(1-2\alpha)(1-\alpha)}}{t+1}+\frac{\frac{\beta}{1-\alpha}}{(t+1)^{2-\alpha}}
+\frac{-\frac{2\beta}{1-2\alpha}}{(t+1)^{2(1-\alpha)}}\bigg].
\nonumber
\end{align}
Hence, we proved \eqref{VarFormula}.

For $0<\alpha<\frac{1}{2}$, from \eqref{VarFormula}, it is easy to check that
\begin{equation}
\lim_{t\rightarrow\infty}\frac{\text{Var}[N_{t}]}{t}=\frac{\beta}{(1-2\alpha)(1-\alpha)}.
\end{equation}

For $\alpha>\frac{1}{2}$, from \eqref{VarFormula}, it is easy to check that
\begin{equation}
\lim_{t\rightarrow\infty}\frac{\text{Var}[N_{t}]}{t^{2\alpha}}=\frac{2\beta}{2\alpha-1}.
\end{equation}
\end{proof}

\begin{proof}[Proof of Theorem \ref{CLTThm}]
Cox and Grimmett \cite{CoxGrimmett} has a central limit theorem for associated random variables.
For a sequence of associated random variables $(X_{n})_{n=1}^{\infty}$, if it satisfies

(i) $\text{Var}[X_{n}]\geq c_{1}$ and $\mathbb{E}[|X_{n}|^{3}]\leq c_{2}$.

(ii) $\sum_{j:|n-j|\geq r}\text{Cov}(X_{j},X_{n})\leq u(r)\rightarrow 0$ as $r\rightarrow\infty$. 

Then, $\frac{S_{n}-\mathbb{E}[S_{n}]}{\sqrt{\text{Var}[S_{n}]}}\rightarrow N(0,1)$ in distribution
as $n\rightarrow\infty$. 

For self-exciting point processes, a new jump will increase the intensity which will help generate more jumps. 
Under the assumption $\lambda(\cdot)$ is increasing, our model is in the class of self-exciting point processes
studied by Kwieci\'{n}ski and Szekli \cite{Kwiecinski} and $(N(n,n+1])_{n=0}^{\infty}$ are associated random variables.

One can use the formulas in Proposition \ref{Expectation} to show (i) directly. Alternatively, we can observe that
for any $T>0$, the compensator of $N[t,t+T]$ is $\int_{t}^{t+T}\lambda_{s}ds$ which converges to $\frac{\beta}{1-\alpha}T$ as $t\rightarrow\infty$. Hence $N[t,t+T]$, $T>0$ converges to a standard 
Poisson process $\bar{N}$ with parameter $\frac{\beta}{1-\alpha}$
as $t\rightarrow\infty$. Thus, $\text{Var}[N[t,t+1]]\rightarrow\text{Var}[\bar{N}[0,1]]$,
$\mathbb{E}[N[t,t+1]^{3}]\rightarrow\mathbb{E}[\bar{N}[0,1]^{3}]$ as $t\rightarrow\infty$.

Note that by Proposition \ref{CovProp}, for any $t>s+1$,
\begin{align}
&\text{Cov}(N[t,t+1],N[s,s+1])
\\
&=\text{Cov}(N_{t+1},N_{s+1})
-\text{Cov}(N_{t+1},N_{s})-\text{Cov}(N_{t},N_{s+1})+\text{Cov}(N_{t},N_{s})
\nonumber
\\
&=(t+2)^{\alpha}\left[\frac{\beta}{(1-2\alpha)(1-\alpha)}(s+2)^{1-\alpha}+\frac{\beta}{1-\alpha}
-\frac{2\beta}{1-2\alpha}(s+2)^{\alpha}\right]
\nonumber
\\
&\qquad
-(t+2)^{\alpha}\left[\frac{\beta}{(1-2\alpha)(1-\alpha)}(s+1)^{1-\alpha}+\frac{\beta}{1-\alpha}
-\frac{2\beta}{1-2\alpha}(s+1)^{\alpha}\right]
\nonumber
\\
&\qquad
-(t+1)^{\alpha}\left[\frac{\beta}{(1-2\alpha)(1-\alpha)}(s+2)^{1-\alpha}+\frac{\beta}{1-\alpha}
-\frac{2\beta}{1-2\alpha}(s+2)^{\alpha}\right]
\nonumber
\\
&\qquad
+(t+1)^{\alpha}\left[\frac{\beta}{(1-2\alpha)(1-\alpha)}(s+1)^{1-\alpha}+\frac{\beta}{1-\alpha}
-\frac{2\beta}{1-2\alpha}(s+1)^{\alpha}\right]
\nonumber
\\
&=\frac{\beta}{(1-2\alpha)(1-\alpha)}[(t+2)^{\alpha}-(t+1)^{\alpha}][(s+2)^{1-\alpha}-(s+1)^{1-\alpha}]
\nonumber
\\
&\qquad
-\frac{2\beta}{1-2\alpha}[(t+2)^{\alpha}-(t+1)^{\alpha}][(s+2)^{\alpha}-(s+1)^{\alpha}]
\nonumber
\\
&\leq\frac{\beta}{(1-2\alpha)(1-\alpha)}[(t+2)^{\alpha}-(t+1)^{\alpha}][(s+2)^{1-\alpha}-(s+1)^{1-\alpha}]
\nonumber
\\
&\leq\frac{\beta}{(1-2\alpha)(1-\alpha)}\frac{\alpha}{(t+1)^{1-\alpha}}\frac{1-\alpha}{(s+1)^{\alpha}}
\nonumber
\\
&\leq\frac{\alpha\beta}{1-2\alpha}\frac{1}{(t+1)^{1-\alpha}}.
\nonumber
\end{align}
This proved (ii). Since $\mathbb{E}[N[t,t+1]]$ is uniformly bounded in $t$, discrete time CLT can 
be replaced by continuous time CLT and we have $\frac{N_{t}-\mathbb{E}[N_{t}]}{\sqrt{\text{Var}[N_{t}]}}\rightarrow N(0,1)$
in distribution as $t\rightarrow\infty$. Finally, by the expressions of $\mathbb{E}[N_{t}]$ and $\text{Var}[N_{t}]$ in
Proposition \ref{Expectation}, we proved \eqref{CLTEqn}.
\end{proof}

\begin{proof}[Proof of Proposition \ref{OneHalf}]
Following the proof of Proposition \ref{Expectation}, for $\alpha=\frac{1}{2}$, $m_{2}(t)=\mathbb{E}[Y_{t}^{2}]$,
\begin{align}
&\frac{dm_{2}(t)}{dt}=\frac{2(\alpha-1)m_{2}(t)}{t+1}+\frac{2\beta^{2}}{1-\alpha}\left[\frac{1}{t+1}-\frac{1}{(t+1)^{2-\alpha}}\right]
\\
&\qquad\qquad
+\frac{\frac{\beta}{1-\alpha}}{(t+1)^{2}}
-\frac{\beta\alpha}{1-\alpha}\frac{1}{(t+1)^{3-\alpha}},
\qquad m_{2}(0)=0.
\nonumber
\end{align}
Consider $m_{2}(t)=\frac{C_{1}}{t+1}+\frac{C_{2}}{(t+1)^{2-\alpha}}+C_{3}+\frac{C_{4}}{(t+1)^{1-\alpha}}+\frac{C_{5}\log(t+1)}{t+1}$
and use the initial condition $m_{2}(0)=0$, we get
\begin{align}
&C_{1}=4\beta^{2}-2\beta,
\\
&C_{2}=\frac{\beta}{1-\alpha}=2\beta,
\nonumber
\\
&C_{3}=\frac{\beta^{2}}{(1-\alpha)^{2}}=4\beta^{2},
\nonumber
\\
&C_{4}=\frac{-2\beta^{2}}{(1-\alpha)^{2}}=-8\beta^{2},
\nonumber
\\
&C_{5}=2\beta.
\nonumber
\end{align}
Therefore,
\begin{align}
\text{Var}[N_{t}]&=(t+1)^{2}[m_{2}(t)-(m_{1}(t))^{2}]
\\
&=(t+1)^{2}\bigg[\frac{4\beta^{2}-2\beta+2\beta\log(t+1)}{t+1}+\frac{2\beta}{(t+1)^{3/2}}+4\beta^{2}
+\frac{-8\beta^{2}}{(t+1)^{1/2}}
\nonumber
\\
&\qquad\qquad
-4\beta^{2}-\frac{4\beta^{2}}{t+1}+\frac{8\beta^{2}}{(t+1)^{1/2}}
\bigg]
\nonumber
\\
&=(t+1)^{2}\left[\frac{2\beta[\log(t+1)-1]}{t+1}+\frac{2\beta}{(t+1)^{3/2}}\right].
\nonumber
\end{align}
\end{proof}

\begin{proof}[Proof of Proposition \ref{One}]
Let $m_{1}(t)=\mathbb{E}[Y_{t}]$ and assume $Y_{0}=0$, then,
\begin{equation}
\frac{dm_{1}(t)}{dt}=\frac{\beta}{t+1},\qquad m_{1}(0)=0,
\end{equation}
which yields that $m_{1}(t)=\beta\log(t+1)$.

Let $m_{2}(t)=\mathbb{E}[Y_{t}^{2}]$. Then, $m_{2}(0)=0$,
\begin{align}
\frac{dm_{2}(t)}{dt}&=\frac{2\beta m_{1}(t)}{t+1}+\frac{\beta+m_{1}(t)}{(t+1)^{2}}
\\
&=\frac{2\beta^{2}\log(t+1)}{t+1}+\frac{\beta+\beta\log(t+1)}{(t+1)^{2}},
\nonumber
\end{align}
which implies that
\begin{equation}
m_{2}(t)=\beta^{2}[\log(t+1)]^{2}-\frac{\beta\log(t+1)}{t+1}+2\beta\left[1-\frac{1}{t+1}\right].
\end{equation}
Therefore,
\begin{align}
\text{Var}[N_{t}]&=(t+1)^{2}[m_{2}(t)-(m_{1}(t))^{2}]
\\
&=(t+1)^{2}\left[-\frac{\beta\log(t+1)}{t+1}+2\beta\left[1-\frac{1}{t+1}\right]\right]
\nonumber
\\
&=-\beta(t+1)\log(t+1)+2\beta t(t+1).
\nonumber
\end{align}
\end{proof}

\begin{proof}[Proof of Corollary \ref{LLNAnother}]
The proof follows from the results in Proposition \ref{Expectation}, Proposition \ref{OneHalf}, Proposition \ref{One}
and Chebychev's inequality.
\end{proof}

\begin{proof}[Proof of Proposition \ref{CovProp}]
For any $t>s$,
\begin{equation}
\mathbb{E}[N_{t}N_{s}]=\mathbb{E}[N_{s}^{2}]+\mathbb{E}\left[N_{s}\int_{s}^{t}\lambda_{u}du\right]
=\mathbb{E}[N_{s}^{2}]+\int_{s}^{t}\left(\beta+\alpha\frac{\mathbb{E}[N_{u}N_{s}]}{u+1}\right)du.
\end{equation}
Let $m(t,s):=\mathbb{E}[N_{t}N_{s}]$. Then
\begin{equation}
\frac{\partial m}{\partial t}=\beta\mathbb{E}[N_{s}]+\frac{\alpha}{t+1}m(t,s),
\qquad m(s,s)=\mathbb{E}[N_{s}^{2}],
\end{equation}
which yields the solution when $\alpha\neq 1$,
\begin{equation}\label{TS}
m(t,s)=(t+1)\frac{\beta\mathbb{E}[N_{s}]}{1-\alpha}
+(t+1)^{\alpha}\frac{\mathbb{E}[N_{s}^{2}]-\frac{\beta\mathbb{E}[N_{s}]}{1-\alpha}(s+1)}{(s+1)^{\alpha}}.
\end{equation}
>From the proofs of Proposition \ref{Expectation}, for $\alpha\notin\{\frac{1}{2},1\}$,
\begin{align}
&\mathbb{E}[N_{t}]=\frac{\beta}{1-\alpha}[(t+1)-(t+1)^{\alpha}],
\\
&\mathbb{E}[N_{t}^{2}]=\frac{\beta}{(1-2\alpha)(1-\alpha)}(t+1)+\frac{\beta}{1-\alpha}(t+1)^{\alpha}
+\frac{\beta^{2}}{(1-\alpha)^{2}}(t+1)^{2}
\nonumber
\\
&\qquad\qquad\qquad
-\frac{2\beta^{2}}{(1-\alpha)^{2}}(t+1)^{\alpha+1}
-\frac{\beta}{1-\alpha}\left[1-\frac{\beta}{1-\alpha}+\frac{1}{1-2\alpha}\right](t+1)^{2\alpha}.
\nonumber
\end{align}
Substituting them into \eqref{TS} and using the idensity 
$\text{Cov}[N_{t},N_{s}]=\mathbb{E}[N_{t}N_{s}]-\mathbb{E}[N_{t}]\mathbb{E}[N_{s}]$, we get
\begin{align}
&\text{Cov}[N_{t},N_{s}]
\\
&=\frac{\beta^{2}}{(1-\alpha)^{2}}(t+1)[(s+1)-(s+1)^{\alpha}]
-\frac{\beta^{2}}{(1-\alpha)^{2}}[(s+1)^{2-\alpha}-(s+1)](t+1)^{\alpha}
\nonumber
\\
&\qquad
+(t+1)^{\alpha}\bigg[\frac{\beta}{(1-2\alpha)(1-\alpha)}(s+1)^{1-\alpha}+\frac{\beta}{1-\alpha}
+\frac{\beta^{2}}{(1-\alpha)^{2}}(s+1)^{2-\alpha}
\nonumber
\\
&\qquad\qquad\qquad
-\frac{2\beta^{2}}{(1-\alpha)^{2}}(s+1)
-\frac{\beta}{1-\alpha}\left[1-\frac{\beta}{1-\alpha}+\frac{1}{1-2\alpha}\right](s+1)^{\alpha}\bigg]
\nonumber
\\
&\qquad\qquad
-\frac{\beta^{2}}{(1-\alpha)^{2}}[(t+1)-(t+1)^{\alpha}][(s+1)-(s+1)^{\alpha}]
\nonumber
\\
&=
(t+1)^{\alpha}\left[\frac{\beta}{(1-2\alpha)(1-\alpha)}(s+1)^{1-\alpha}+\frac{\beta}{1-\alpha}
-\frac{2\beta}{1-2\alpha}(s+1)^{\alpha}\right].
\nonumber
\end{align}
For $\alpha=\frac{1}{2}$, from Proposition \ref{OneHalf},
\begin{align*}
&\mathbb{E}[N_{t}]=2\beta[(t+1)-(t+1)^{\frac{1}{2}}],
\\
&\mathbb{E}[N_{t}^{2}]=(4\beta^{2}-2\beta)(t+1)+2\beta(t+1)^{\frac{1}{2}}+4\beta^{2}(t+1)^{2}
\\
&\qquad\qquad\qquad
-8\beta^{2}(t+1)^{\frac{3}{2}}+2\beta(t+1)\log(t+1).
\end{align*}
Hence, substituting these into \eqref{TS}, we get
\begin{align}
\text{Cov}[N_{t},N_{s}]
&=m(t,s)-\mathbb{E}[N_{t}]\mathbb{E}[N_{s}]
\\
&=(t+1)2\beta\mathbb{E}[N_{s}]
+(t+1)^{\frac{1}{2}}\frac{\mathbb{E}[N_{s}^{2}]-2\beta\mathbb{E}[N_{s}](s+1)}{(s+1)^{\frac{1}{2}}}
\nonumber
\\
&\qquad\qquad
-4\beta^{2}[(t+1)-(t+1)^{\frac{1}{2}}][(s+1)-(s+1)^{\frac{1}{2}}]
\nonumber
\\
&=2\beta[-(t+1)^{\frac{1}{2}}(s+1)^{\frac{1}{2}}+(t+1)^{\frac{1}{2}}+(t+1)^{\frac{1}{2}}(s+1)^{\frac{1}{2}}\log(s+1)].
\nonumber
\end{align}
\end{proof}

\begin{proof}[Proof of Proposition \ref{CovOneProp}]
Let $m(t,s)=\mathbb{E}[N_{t}N_{s}]$. Following the proofs in Proposition \ref{CovProp},
\begin{equation}
\frac{\partial m}{\partial t}=\beta\mathbb{E}[N_{s}]+\frac{1}{t+1}m(t,s),
\qquad m(s,s)=\mathbb{E}[N_{s}^{2}],
\end{equation}
which yields the solution
\begin{equation}\label{TSII}
m(t,s)=\beta\mathbb{E}[N_{s}](t+1)\log(t+1)
+(t+1)\frac{\mathbb{E}[N_{s}^{2}]-\beta\mathbb{E}[N_{s}](s+1)\log(s+1)}{s+1}.
\end{equation}
>From Proposition \ref{One},
\begin{align}
&\mathbb{E}[N_{t}]=\beta(t+1)\log(t+1),
\\
&\mathbb{E}[N_{t}^{2}]=-\beta(t+1)\log(t+1)+2\beta t(t+1)+\beta^{2}[(t+1)\log(t+1)]^{2}.
\nonumber
\end{align}
Substituting this into \eqref{TSII} and using the idensity 
$\text{Cov}[N_{t},N_{s}]=\mathbb{E}[N_{t}N_{s}]-\mathbb{E}[N_{t}]\mathbb{E}[N_{s}]$, we get
\begin{align}
&\text{Cov}[N_{t},N_{s}]
\\
&=\beta\mathbb{E}[N_{s}](t+1)\log(t+1)
+(t+1)\frac{\mathbb{E}[N_{s}^{2}]-\beta\mathbb{E}[N_{s}](s+1)\log(s+1)}{s+1}
\nonumber
\\
&\qquad\qquad
-\beta^{2}[(t+1)\log(t+1)][(s+1)\log(s+1)]
\nonumber
\\
&=\frac{t+1}{s+1}\left[-\beta(s+1)\log(s+1)+2\beta s(s+1)+\beta^{2}[(s+1)\log(s+1)]^{2}\right]
\nonumber
\\
&\qquad\qquad
-\beta(t+1)\log(s+1)\beta(s+1)\log(s+1)
\nonumber
\\
&=-\beta(t+1)\log(s+1)+2\beta s(t+1).
\nonumber
\end{align}
\end{proof}

\begin{proof}[Proof of Proposition \ref{GammaLimit}]
Since
\begin{equation}
\mathbb{E}[N_{t}]=\mathbb{E}\left[\int_{0}^{t}\lambda_{s}ds\right]
=\int_{0}^{t}\frac{\alpha(\mathbb{E}[N_{s}]+\gamma)}{s+1}ds.
\end{equation}
By letting $g(t):=\mathbb{E}[N_{t}]$, it satisfies the ODE
\begin{equation}
g'(t)=\frac{\alpha(g(t)+\gamma)}{t+1},\qquad g(0)=0,
\end{equation}
which yields the solution $g(t)=\gamma[(t+1)^{\alpha}-1]$.
This is consistent with \eqref{ExampleFormula} and the variance of a negative binomial distribution.
Next, let $h(t)=\mathbb{E}[N_{t}^{2}]$. Then
\begin{equation}
d(N_{t}^{2})=(2N_{t-}+1)dN_{t},
\end{equation}
and hence after taking expectations, $h(0)=0$ and
\begin{align}
h'(t)&=2\alpha\frac{h(t)}{t+1}+\frac{2\gamma\alpha+\alpha}{t+1}g(t)+\frac{\gamma\alpha}{t+1}
\\
&=\frac{2\alpha h(t)}{t+1}+\frac{\frac{2\gamma^{2}\alpha^{2}}{\alpha}+\gamma\alpha}{(t+1)^{1-\alpha}}-\frac{\frac{2\gamma^{2}\alpha^{2}}{\alpha}}{t+1},
\nonumber
\end{align}
which yields the solution
\begin{equation}
h(t)=\frac{\gamma^{2}\alpha^{2}}{\alpha^{2}}
-\left[\frac{2\gamma^{2}\alpha^{2}}{\alpha^{2}}+\frac{\gamma\alpha}{\alpha}\right](t+1)^{\alpha}
+\left[\frac{\gamma^{2}\alpha^{2}}{\alpha^{2}}+\frac{\gamma\alpha}{\alpha}\right](t+1)^{2\alpha}.
\end{equation}
Hence
\begin{equation}
\text{Var}[N_{t}]=\gamma[(t+1)^{2\alpha}-(t+1)^{\alpha}].
\end{equation}
This is consistent with \eqref{ExampleFormula} and the variance of a negative binomial distribution.
Furthermore, by \eqref{ConditionalFormula} and the properties of negative binomial distributions
\begin{align}
&\text{Cov}[N_{t},N_{s}]
\\
&=\mathbb{E}[\mathbb{E}[N_{t}|N_{s}]N_{s}]-\mathbb{E}[N_{t}]\mathbb{E}[N_{s}]
\nonumber
\\
&=\mathbb{E}\left[(N_{s}+\gamma)\left[\left(\frac{t+1}{s+1}\right)^{\alpha}-1\right]N_{s}+N_{s}^{2}\right]
-\mathbb{E}[N_{t}]\mathbb{E}[N_{s}]
\nonumber
\\
&=\left[\left(\frac{t+1}{s+1}\right)^{\alpha}-1\right]\left[\gamma[(s+1)^{2\alpha}-(s+1)^{\alpha}]
+\gamma^{2}[(s+1)^{\alpha}-1]^{2}+\gamma[(s+1)^{\alpha}-1]\right]
\nonumber
\\
&\qquad
+\gamma[(s+1)^{2\alpha}-(s+1)^{\alpha}]+\gamma^{2}[(s+1)^{\alpha}-1]^{2}
-\gamma^{2}[(t+1)^{\alpha}-1][(s+1)^{\alpha}-1]
\nonumber
\\
&=[(s+1)^{\alpha}-1]\left[\gamma(t+1)^{\alpha}+(\gamma-\gamma^{2})\left[\frac{(t+1)^{\alpha}}{(s+1)^{\alpha}}-1\right]\right].
\nonumber
\end{align}
Moreover,
\begin{equation}
d\left(\frac{N_{t}+\gamma}{(t+1)^{\alpha}}\right)
=-\alpha\frac{N_{t}+\gamma}{(t+1)^{\alpha+1}}dt+\frac{dN_{t}}{(t+1)^{\alpha}}
=\frac{dM_{t}}{(t+1)^{\alpha}}.
\end{equation}
Hence, $\frac{N_{t}+\gamma}{(t+1)^{\alpha}}$ is a martingale and from \eqref{alphaVar} 
we have that $\sup_{t>0}\mathbb{E}\left[\frac{N_{t}+\gamma}{(t+1)^{\alpha}}\right]<\infty$. Therefore,
by martingale convergence theorem, 
$\frac{N_{t}}{t^{\alpha}}\rightarrow\chi(\alpha,\gamma)$,
a.s. and in $L^{2}(\mathbb{P})$, for some random variable 
$\chi(\alpha,\gamma)$ which is finite a.s. and in $L^{2}(\mathbb{P})$ and possibly depends on 
parameters $\alpha$ and $\gamma$.
Finally, by \eqref{ExampleFormula} and the formula for the Laplace transform 
of negative binomial distribution, for any $\theta>0$, 
\begin{equation}
\mathbb{E}\left[e^{-\theta\frac{N_{t}}{t^{\alpha}}}\right]
=\left(\frac{\frac{1}{(t+1)^{\alpha}}}{1-\left(1-\frac{1}{(t+1)^{\alpha}}\right)
e^{-\frac{\theta}{t^{\alpha}}}}\right)^{\gamma}
\rightarrow\left(\frac{1}{1+\theta}\right)^{\gamma},
\end{equation}
as $t\rightarrow\infty$. Hence $\chi(\alpha,\gamma)$ is independent of $\alpha$
and follows a gamma distribution with shape $\gamma$ and scale $1$. 
\end{proof}

\begin{proof}[Proof of Proposition \ref{ExplosionProp}]
For any $T>0$, $\lambda_{t}\geq\lambda(\frac{N_{t-}+\gamma}{T+1})$ on $[0,T]$. 
Comparing with the pure-birth process, see e.g. Feller \cite{Feller}, it becomes clear
that $\mathbb{P}(N(0,T]=\infty)>0$. Moreover, to see $\mathbb{P}(\tau<\infty)<1$, it suffices
to notice that
\begin{equation}
\mathbb{P}(\tau=\infty)\geq\mathbb{P}(N(0,\infty)=\infty)
=e^{-\int_{0}^{\infty}\lambda(\frac{\gamma}{s+1})ds}\in(0,\infty).
\end{equation}
\end{proof}

\subsection{Proof of Results in Section \ref{LargeInitialSection}}

\begin{proof}[Proof of Proposition \ref{LargeInitial}]
Let us assume first that $\lambda(z)=\alpha z$. Then, we know that
\begin{equation}
\frac{N_{t}+\gamma}{(t+1)^{\alpha}}-\gamma=\int_{0}^{t}\frac{dM_{s}}{(s+1)^{\alpha}}
\end{equation}
is a martingale. Therefore, for any $\epsilon>0$, using $\mathbb{E}[N_{t}]=\gamma[(t+1)^{\alpha}-1]$ from Proposition \ref{GammaLimit}
\begin{align}
\mathbb{P}\left(\sup_{0\leq s\leq t}\left|\frac{N_{s}+\gamma}{(s+1)^{\alpha}}-\gamma\right|\geq\epsilon\gamma\right)
&\leq\frac{\mathbb{E}\left[\left(\int_{0}^{t}\frac{dM_{s}}{(s+1)^{\alpha}}\right)^{2}\right]}{\epsilon^{2}\gamma^{2}}
\\
&=\frac{\int_{0}^{t}\frac{\mathbb{E}[\lambda_{s}]}{(s+1)^{\alpha}}ds}{\epsilon^{2}\gamma^{2}}
\nonumber
\\
&=\frac{\int_{0}^{t}\frac{\alpha\gamma}{s+1}ds}{\epsilon^{2}\gamma^{2}}
\nonumber
\\
&\rightarrow 0,\nonumber
\end{align}
as $\gamma\rightarrow\infty$. Hence,
\begin{align}
\mathbb{P}\left(\sup_{0\leq s\leq t}\left|\frac{N_{s}}{\gamma}-[(s+1)^{\alpha}-1]\right|\geq\epsilon\right)
&=\mathbb{P}\left(\sup_{0\leq s\leq t}\left|\frac{N_{s}+\gamma}{(s+1)^{\alpha}}-\gamma\right|(s+1)^{\alpha}\geq\epsilon\gamma\right)
\\
&
\leq\mathbb{P}\left(\sup_{0\leq s\leq t}\left|\frac{N_{s}+\gamma}{(s+1)^{\alpha}}-\gamma\right|\geq\frac{\epsilon\gamma}{(t+1)^{\alpha}}\right)
\nonumber
\\
&\rightarrow\infty,
\nonumber
\end{align}
as $\gamma\rightarrow\infty$. If $\lim_{z\rightarrow\infty}\frac{\lambda(z)}{z}=\alpha$, then for any $\delta>0$,
there exists $K$ so that for any $z\geq K$, $(\alpha-\delta)z\leq\lambda(z)\leq(\alpha+\delta)z$. 
Uniformly for $0\leq s\leq t$, $\frac{N_{s-}+\gamma}{s+1}\geq\frac{\gamma}{t+1}\geq K$ for any $\gamma\geq K(t+1)$.
Now using the results for $\lambda(z)=(\alpha+\delta)z$ and $\lambda(z)=(\alpha-\delta)z$ and let $\delta\rightarrow 0$, we proved Proposition \ref{LargeInitial}.
\end{proof}

\begin{proof}[Proof of Proposition \ref{SublinearInitial}]
For any $\epsilon>0$ and fixed $t>0$, for suffciently large $\gamma$,
$(\alpha-\epsilon)z^{\beta}\leq\lambda(z)\leq(\alpha+\epsilon)z^{\beta}$ for any $z\geq\frac{\gamma}{t+1}$. 
Since it holds for any $\epsilon>0$, to prove Proposition \ref{SublinearInitial}, it suffices
to consider the case $\lambda(z)=\alpha z^{\beta}$. Without loss of generality, let us take $\alpha=1$.
Let us use the Poisson embedding. 
Let $N^{(0)}$ be the Poisson process with intensity $\lambda(\frac{\gamma}{t+1})=\frac{\gamma^{\beta}}{(t+1)^{\beta}}$
and the compensator
\begin{equation}
\int_{0}^{t}\lambda\left(\frac{\gamma}{s+1}\right)ds
=\int_{0}^{t}\frac{\gamma^{\beta}}{(s+1)^{\beta}}ds=\frac{\gamma^{\beta}}{1-\beta}[(t+1)^{1-\beta}-1].
\end{equation}
It is easy to show that
\begin{equation}
\sup_{0\leq s\leq t}\left|\frac{N_{s}^{(0)}}{\gamma^{\beta}}-\frac{1}{1-\beta}[(s+1)^{1-\beta}-1]\right|\rightarrow 0,
\end{equation}
in probability as $\gamma\rightarrow\infty$.

Conditional on $N^{(0)}$, let $N^{(1)}$ be the inhomogeneous Poisson process with intensity
\begin{equation}
\lambda\left(\frac{N^{(0)}_{t-}+\gamma}{t+1}\right)-\lambda\left(\frac{\gamma}{t+1}\right),
\end{equation}
at time $t$. Inductively, conditional on $N^{(0)},N^{(1)},\ldots,N^{(k)}$, $N^{(k+1)}$ is an inhomogeneous Poisson process
with intensity
\begin{equation}
\lambda\left(\frac{N^{(0)}_{t-}+N^{(1)}_{t-}+\cdots N^{(k)}_{t-}+\gamma}{t+1}\right)
-\lambda\left(\frac{N^{(0)}_{t-}+N^{(1)}_{t-}+\cdots N^{(k-1)}_{t-}+\gamma}{t+1}\right),
\end{equation}
at time $t$. By the mean value theorem, and the assumption $0<\beta<1$,
\begin{align}
&\lambda\left(\frac{N^{(0)}_{s-}+N^{(1)}_{s-}+\cdots N^{(k)}_{s-}+\gamma}{s+1}\right)
-\lambda\left(\frac{N^{(0)}_{s-}+N^{(1)}_{s-}+\cdots N^{(k-1)}_{s-}+\gamma}{s+1}\right)
\\
&\leq\beta\left(\frac{\gamma}{s+1}\right)^{\beta-1}N^{(0)}_{s}
\nonumber
\\
&\leq\beta\left(\frac{\gamma}{t+1}\right)^{\beta-1}N^{(0)}_{t}.
\nonumber
\end{align}
Therefore, by induction,
\begin{align}
\mathbb{E}[N_{t}^{(k+1)}]
&\leq\beta t\left(\frac{\gamma}{t+1}\right)^{\beta-1}\mathbb{E}[N^{(k)}_{t}]
\\
&\leq\left(\beta t\left(\frac{\gamma}{t+1}\right)^{\beta-1}\right)^{k+1}\mathbb{E}[N^{(0)}_{t}].
\nonumber
\end{align}
For fixed $t$, for sufficiently large $\gamma$, we have $t\left(\frac{\gamma}{t+1}\right)^{\beta-1}\leq 1$.
Hence, $\mathbb{E}[N_{t}^{(k+1)}]\leq\beta^{k+1}\mathbb{E}[N^{(0)}_{t}]$ 
and $\sum_{k=0}^{\infty}\mathbb{E}[N^{(k)}_{t}]\leq\frac{1}{1-\beta}\mathbb{E}[N^{(0)}_{t}]<\infty$
$N_{t}=\sum_{k=0}^{\infty}N_{t}^{(k)}$ is well defined and coincides with the self-exciting point process in our model from Poisson embedding.
Moreover,
\begin{align}
\mathbb{E}\left[\sum_{k=1}^{\infty}N_{t}^{(k)}\right]
&\leq\sum_{k=1}^{\infty}\left(\beta\left(\frac{\gamma}{t+1}\right)^{\beta-1}\right)^{k}\mathbb{E}[N^{(0)}_{t}]
\\
&\leq\beta\left(\frac{\gamma}{t+1}\right)^{\beta-1}\frac{\gamma^{\beta}}{1-\beta}[(t+1)^{1-\beta}-1].
\nonumber
\end{align}
Therefore, we conclude that $\frac{\sum_{k=1}^{\infty}N_{t}^{(k)}}{\gamma^{\beta}}\rightarrow 0$ in probability as $\gamma\rightarrow\infty$.
Hence, we proved the desired result.
\end{proof}

\subsection{Proofs of Results in Section \ref{TailSection}}

\begin{proof}[Proof of Theorem \ref{MarginalProp}]
(i) The identity \eqref{MasterFormula} holds from the definition of our model.
The integrand is the infinitesimal probability that there are precisely $k$ jumps on the time
interval $[0,t]$ that occurs at $0<t_{1}<t_{2}<\cdots<t_{k}$.

(ii) This is a direct consequence of (i).

(iii) When $\lambda(z)=\alpha z$, 
\begin{align}
\mathbb{P}(N_{t}=k)
&=\idotsint_{0<t_{1}<t_{2}<\cdots<t_{k}<t}
\alpha^{k}\prod_{j=1}^{k}\frac{\gamma+j-1}{t_{j}+1}
\\
&\qquad\qquad
\cdot
e^{-\int_{0}^{t_{1}}\frac{\alpha\gamma}{s+1}ds
-\int_{t_{1}}^{t_{2}}\frac{\alpha(\gamma+1)}{s+1}ds
-\cdots-\int_{t_{k}}^{t}\frac{\alpha(\gamma+k)}{s+1}ds}
dt_{1}dt_{2}\cdots dt_{k}
\nonumber
\\
&=\alpha^{k}\gamma(\gamma+1)\cdots(\gamma+k-1)
\idotsint_{0<t_{1}<t_{2}<\cdots<t_{k}<t}
\prod_{j=1}^{k}\frac{1}{t_{j}+1}
\nonumber
\\
&\quad
\exp\bigg\{-\alpha\gamma\log(t_{1}+1)+\alpha(\gamma+1)\log(t_{1}+1)-\alpha(\gamma+1)\log(t_{2}+1)
\nonumber
\\
&\quad\quad
\cdots
+\alpha(\gamma+k)\log(t_{k}+1)-\alpha(\gamma+k)\log(t+1)\bigg\}dt_{1}\cdots dt_{k}
\nonumber
\\
&=\alpha^{k}\gamma(\gamma+1)\cdots(\gamma+k-1)\frac{1}{(t+1)^{\alpha(\gamma+k)}}
\nonumber
\\
&\qquad\qquad
\idotsint_{0<t_{1}<t_{2}<\cdots<t_{k}<t}
\prod_{j=1}^{k}\frac{1}{(t_{j}+1)^{1-\alpha}}dt_{1}\cdots dt_{k}
\nonumber
\\
&=\frac{1}{k!}\frac{\gamma(\gamma+1)\cdots(\gamma+k-1)}{(t+1)^{\alpha(\gamma+k)}}
[(t+1)^{\alpha}-1]^{k}
\nonumber
\\
&=\binom{k+\gamma-1}{k}\left(1-\frac{1}{(t+1)^{\alpha}}\right)^{k}
\left(\frac{1}{(t+1)^{\alpha}}\right)^{\gamma}.
\nonumber
\end{align}
In other words, $N_{t}$ follows a negative binomial distribution. Similarly, 
\begin{align}
&\mathbb{P}(N_{t}=k+m|N_{s}=m)
\\
&=\idotsint_{s<t_{1}<t_{2}<\cdots<t_{k}<t}
\alpha^{k}\prod_{j=1}^{k}\frac{\gamma+m+j-1}{t_{j}+1}
\nonumber
\\
&\qquad\qquad
\cdot
e^{-\int_{s}^{t_{1}}\frac{\alpha(\gamma+m)}{s+1}ds
-\int_{t_{1}}^{t_{2}}\frac{\alpha(\gamma+m+1)}{s+1}ds
-\cdots-\int_{t_{k}}^{t}\frac{\alpha(\gamma+m+k)}{s+1}ds}
dt_{1}dt_{2}\cdots dt_{k}
\nonumber
\\
&=\binom{k+m+\gamma-1}{k}\left(1-\left(\frac{s+1}{t+1}\right)^{\alpha}\right)^{k}
\left(\left(\frac{s+1}{t+1}\right)^{\alpha}\right)^{\gamma+m}.
\nonumber
\end{align}
\end{proof}

\begin{proof}[Proof of Theorem \ref{LinearTail}]
For any $\epsilon>0$, there exists a constant $M(\epsilon)$ so that 
for any $z\geq M(\epsilon)$, $(\alpha-\epsilon)z\leq\lambda(z)\leq(\alpha+\epsilon)z$.
Therefore, there exists some constant $C_{1}$ and $C_{2}$ that depend on $\epsilon$, $\gamma$ and $t$ so that for any $k$
\begin{equation}
(\alpha-\epsilon)^{k}C_{1}\leq\prod_{j=1}^{k}\frac{\lambda\left(\frac{\gamma+j-1}{t_{j}+1}\right)}{\frac{\gamma+j-1}{t_{j}+1}}
\leq(\alpha+\epsilon)^{k}C_{2}.
\end{equation}
And there also exist some $C_{3}$ and $C_{4}$ that may depend on $\epsilon$, $\gamma$ and $t$ so that for any $0<t_{1}<\cdots<t_{k}<t$,
\begin{align}
&-C_{3}-\int_{0}^{t_{1}}\frac{(\alpha+\epsilon)\gamma}{s+1}ds
-\int_{t_{1}}^{t_{2}}\frac{(\alpha+\epsilon)(\gamma+1)}{s+1}ds
-\cdots-\int_{t_{k}}^{t}\frac{(\alpha+\epsilon)(\gamma+k)}{s+1}ds
\\
&\leq-\int_{0}^{t_{1}}\lambda\left(\frac{\gamma}{s+1}\right)ds
-\int_{t_{1}}^{t_{2}}\lambda\left(\frac{\gamma+1}{s+1}\right)ds
-\cdots-\int_{t_{k}}^{t}\lambda\left(\frac{\gamma+k}{s+1}\right)ds
\nonumber
\\
&\leq
C_{4}-\int_{0}^{t_{1}}\frac{(\alpha-\epsilon)\gamma}{s+1}ds
-\int_{t_{1}}^{t_{2}}\frac{(\alpha-\epsilon)(\gamma+1)}{s+1}ds
-\cdots-\int_{t_{k}}^{t}\frac{(\alpha-\epsilon)(\gamma+k)}{s+1}ds.
\nonumber
\end{align}
Hence, from the proof of Theorem \ref{MarginalProp}, we have
\begin{align}
&C_{1}e^{-C_{3}}\left(\frac{\alpha-\epsilon}{\alpha+\epsilon}\right)^{k}
\binom{k+\gamma-1}{k}\left(1-\frac{1}{(t+1)^{\alpha+\epsilon}}\right)^{k}
\left(\frac{1}{(t+1)^{\alpha+\epsilon}}\right)^{\gamma}
\\
&\leq\mathbb{P}(N_{t}=k)\leq C_{2}e^{C_{4}}\left(\frac{\alpha+\epsilon}{\alpha-\epsilon}\right)^{k}
\binom{k+\gamma-1}{k}\left(1-\frac{1}{(t+1)^{\alpha-\epsilon}}\right)^{k}
\left(\frac{1}{(t+1)^{\alpha-\epsilon}}\right)^{\gamma}.
\nonumber
\end{align}
Since it holds for any $\epsilon>0$, we proved \eqref{TailLimit}.
\end{proof}

\begin{proof}[Proof of Theorem \ref{SublinearTail}]
The results for the case $\lim_{z\rightarrow\infty}\frac{\lambda(z)}{z^{\beta}}=\alpha$
can be reduced to the case $\lambda(z)=\alpha z^{\beta}$ by following the similar arguments as in the proof of Theorem \ref{LinearTail}.
So let us assume that $\lambda(z)=\alpha z^{\beta}$.
\begin{align}
\mathbb{P}(N_{t}=k)
&=\idotsint_{0<t_{1}<t_{2}<\cdots<t_{k}<t}
\alpha^{k}\prod_{j=1}^{k}\left(\frac{\gamma+j-1}{t_{j}+1}\right)^{\beta}
\\
&\qquad\qquad
\cdot
e^{-\int_{0}^{t_{1}}\frac{\alpha\gamma^{\beta}}{(s+1)^{\beta}}ds
-\int_{t_{1}}^{t_{2}}\frac{\alpha(\gamma+1)^{\beta}}{(s+1)^{\beta}}ds
-\cdots-\int_{t_{k}}^{t}\frac{\alpha(\gamma+k)^{\beta}}{(s+1)^{\beta}}ds}
dt_{1}dt_{2}\cdots dt_{k}
\nonumber
\\
&
\leq\idotsint_{0<t_{1}<t_{2}<\cdots<t_{k}<t}
\alpha^{k}\prod_{j=1}^{k}\left(\gamma+j-1\right)^{\beta}
dt_{1}dt_{2}\cdots dt_{k}
\nonumber
\\
&=\alpha^{k}t^{k}\frac{1}{k!}\prod_{j=1}^{k}\left(\gamma+j-1\right)^{\beta}.
\nonumber
\end{align}
Therefore,
\begin{equation}
\limsup_{\ell\rightarrow\infty}\frac{1}{\ell\log\ell}\log\mathbb{P}(N_{t}\geq\ell)\leq-(1-\beta).
\end{equation}
On the other hand,
\begin{align}
\mathbb{P}(N_{t}=k)
&=\idotsint_{0<t_{1}<t_{2}<\cdots<t_{k}<t}
\alpha^{k}\prod_{j=1}^{k}\left(\frac{\gamma+j-1}{t_{j}+1}\right)^{\beta}
\\
&\qquad\qquad
\cdot
e^{-\int_{0}^{t_{1}}\frac{\alpha\gamma^{\beta}}{(s+1)^{\beta}}ds
-\int_{t_{1}}^{t_{2}}\frac{\alpha(\gamma+1)^{\beta}}{(s+1)^{\beta}}ds
-\cdots-\int_{t_{k}}^{t}\frac{\alpha(\gamma+k)^{\beta}}{(s+1)^{\beta}}ds}
dt_{1}dt_{2}\cdots dt_{k}
\nonumber
\\
&\geq\alpha^{k}t^{k}\frac{1}{k!}\prod_{j=1}^{k}\left(\frac{\gamma+j-1}{t+1}\right)^{\beta}e^{-\alpha (\gamma+k)^{\beta}t}.
\nonumber
\end{align}
Therefore,
\begin{equation}
\liminf_{\ell\rightarrow\infty}\frac{1}{\ell\log\ell}\log\mathbb{P}(N_{t}\geq\ell)\geq-(1-\beta),
\end{equation}
and we proved the desired result.
\end{proof}

\section{Conclusion and Open Problems}\label{OpenSection}

In this paper, we studied a class of self-exciting point processes. We proved that the limit in the law of large numbers
is a fixed point of the rate function. When the rate function is linear, explicit formulas were obtained for the mean, variance
and covariance. Central limit theorem and large deviations were also studied. Finally, for a fixed time interval, we obtain
the asymptotics for the tail probabilities. Here is a list of open problems that are interesting to investigate in the future.
\begin{itemize}
\item
When there are more than one fixed point of $x=\lambda(x)$, say there are exactly two stable fixed points $x_{1}<x_{2}$, 
we made a plot of the probability $p_{1}$ and $p_{2}$ that the process $\frac{N_{t}}{t}$ converges to $x_{1}$ and $x_{2}$ respectively
as a function of the initial condition $\gamma$. From Figure \ref{p1p2}, the simulations suggest that $p_{1}$, $p_{2}$ are monotonic
in $\gamma$. Is that always true? Can we compute $p_{1}(\gamma),p_{2}(\gamma)$ analytically or at least obtain
asymptotics for $\gamma\rightarrow 0^{+}$ and $\gamma\rightarrow\infty$?
\item
So far, we have concentrated on the case when $x=\lambda(x)$ has finitely many fixed points. It is natural to ask
what if there are infinitely many fixed points, or more precisely, what if the Lebesgue measure of the set
of fixed points is positive, then, what will be the limiting distribution of $\frac{N_{t}}{t}$ like
as $t\rightarrow\infty$?. In Figure \ref{lambdaPiecewise}, we consider a piecewise $\lambda(x)$ that
coincides with $x$ on the interval $[2,3]$ and $[4.5,5]$ and Figure \ref{samplePathPiecewise} illustrates the limiting set 
of $\frac{N_{t}}{t}$ as time $t\rightarrow\infty$. Figure \ref{samplePathPiecewise} suggests
that the limiting set is supported on $[2,3]$ and $[4.5,5]$. 

\begin{figure}[H]
\begin{center}
\includegraphics[scale=0.55]{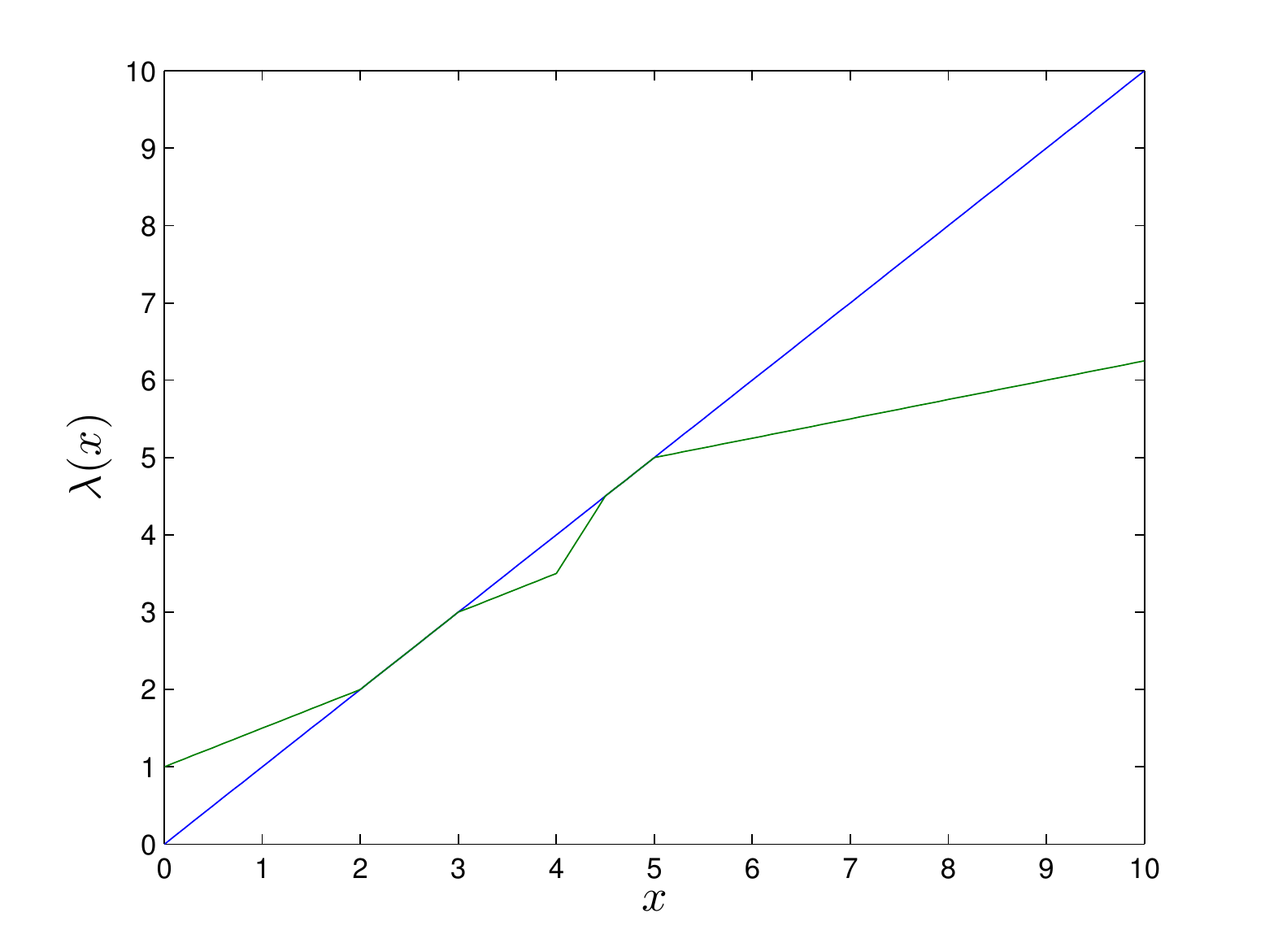}
\caption{We consider a piecewise function $\lambda(x)$ defined as $1+0.5x$ for $x<2$, $x$ for $2\leq x<3$, $1.5+0.5x$
for $3\leq x<4$, $-4.5+2x$ for $4\leq x<4.5$, $x$ for $4.5\leq x<5$ and $3.75+0.25x$ for $x\geq 5$. The set
of the fixed points of $x=\lambda(x)$ is $[2,3]\cup[4.5,5]$.}
\label{lambdaPiecewise}
\end{center}
\end{figure}

\begin{figure}[H]
\begin{center}
\includegraphics[scale=0.70]{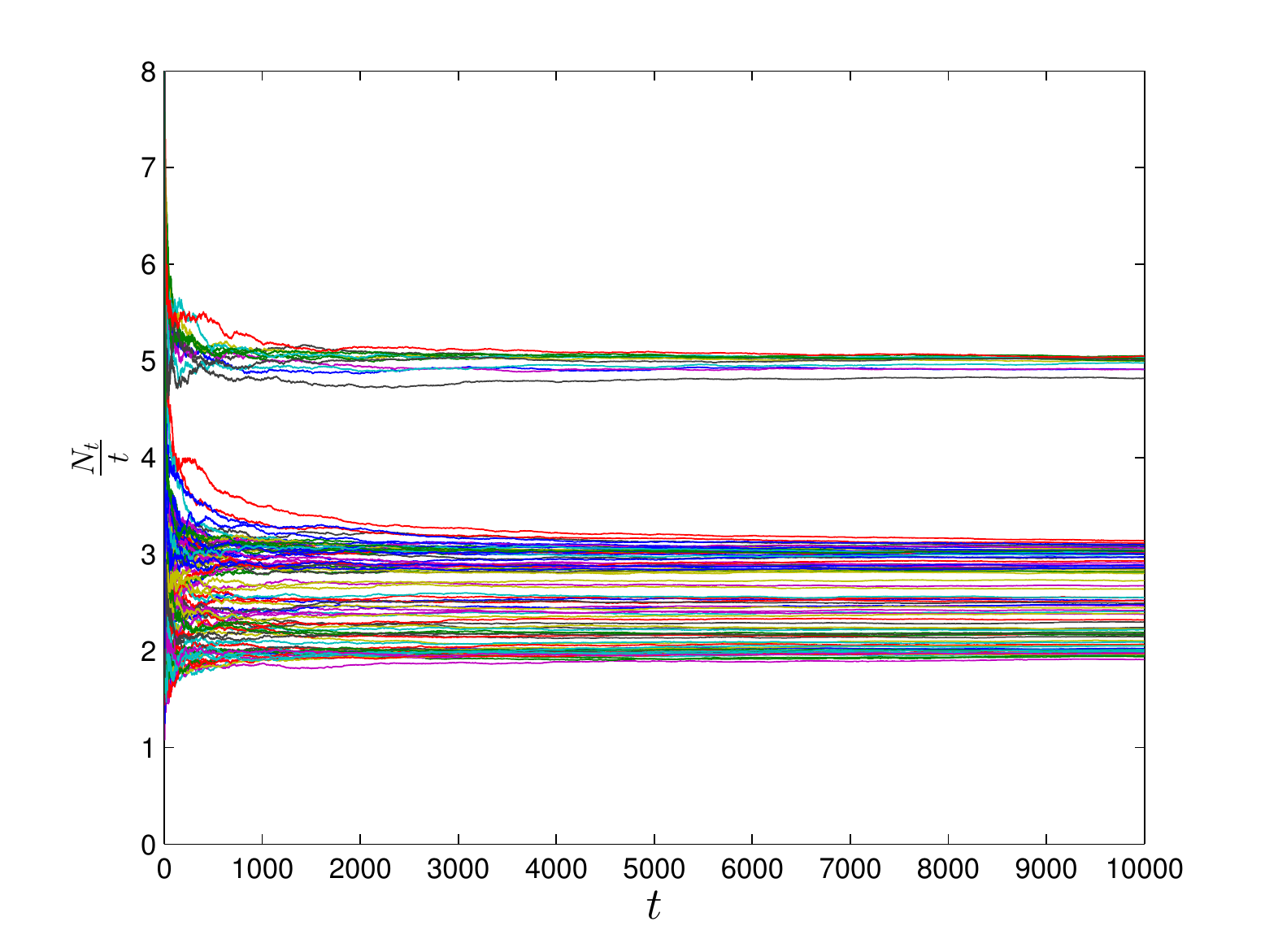}
\caption{We choose the initial starting point as $3.5$. The function $\lambda(x)$ is defined in Figure \ref{lambdaPiecewise}. 
We simulate 100 sample paths and the illustration suggests
that the limiting set of $\frac{N_{t}}{t}$ as time $t\rightarrow\infty$ is supported on $[2,3]$ and $[4.5,5]$.}
\label{samplePathPiecewise}
\end{center}
\end{figure}

\item
Sometimes, a fixed point of $x=\lambda(x)$ can be neither stable or unstable. It is possible to have a saddle point, i.e., stable
from one side and unstable from the other. Figure \ref{lambdaSemi} gives such an example in which $\lambda(x)$ is piecewise linear
and there is a stable fixed point at $x=5$ and two saddle points at $x=2$ and $x=6.5$. Can we analyze this situation?

\begin{figure}[H]
\begin{center}
\includegraphics[scale=0.55]{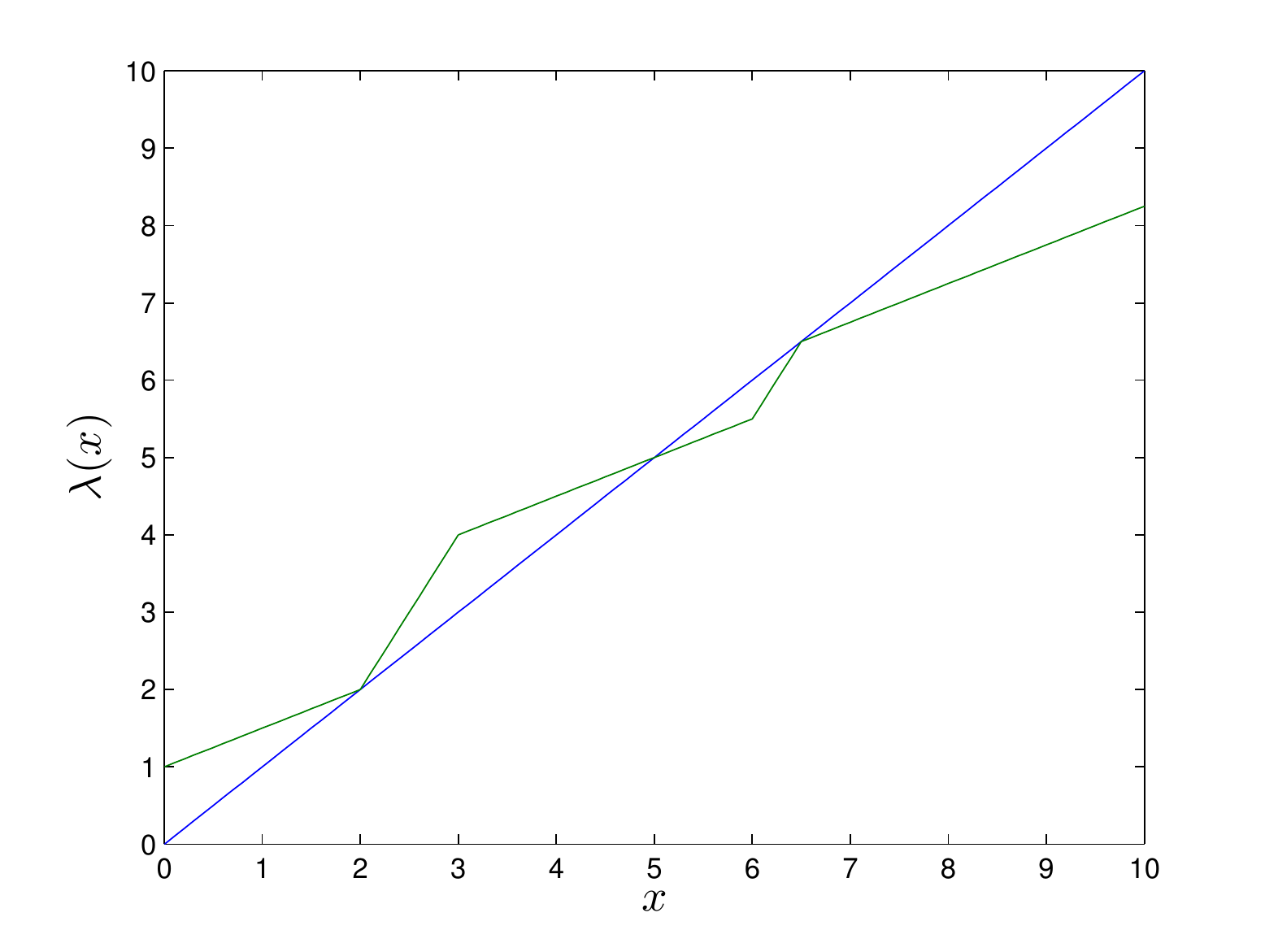}
\caption{We consider a piecewise function $\lambda(x)$ defined as $1+0.5x$ for $x<2$, $-2+2x$ for $2\leq x<3$, $2.5+0.5x$
for $3\leq x<6$, $-6.5+2x$ for $6\leq x<6.5$, and $3.25+0.5x$ for $x\geq 6.5$. The set
of the fixed points consist of a stable fixed point at $x=5$ and two saddle points at $x=2$ and $x=6.5$.}
\label{lambdaSemi}
\end{center}
\end{figure}

\begin{figure}[H]
\begin{center}
\includegraphics[scale=0.70]{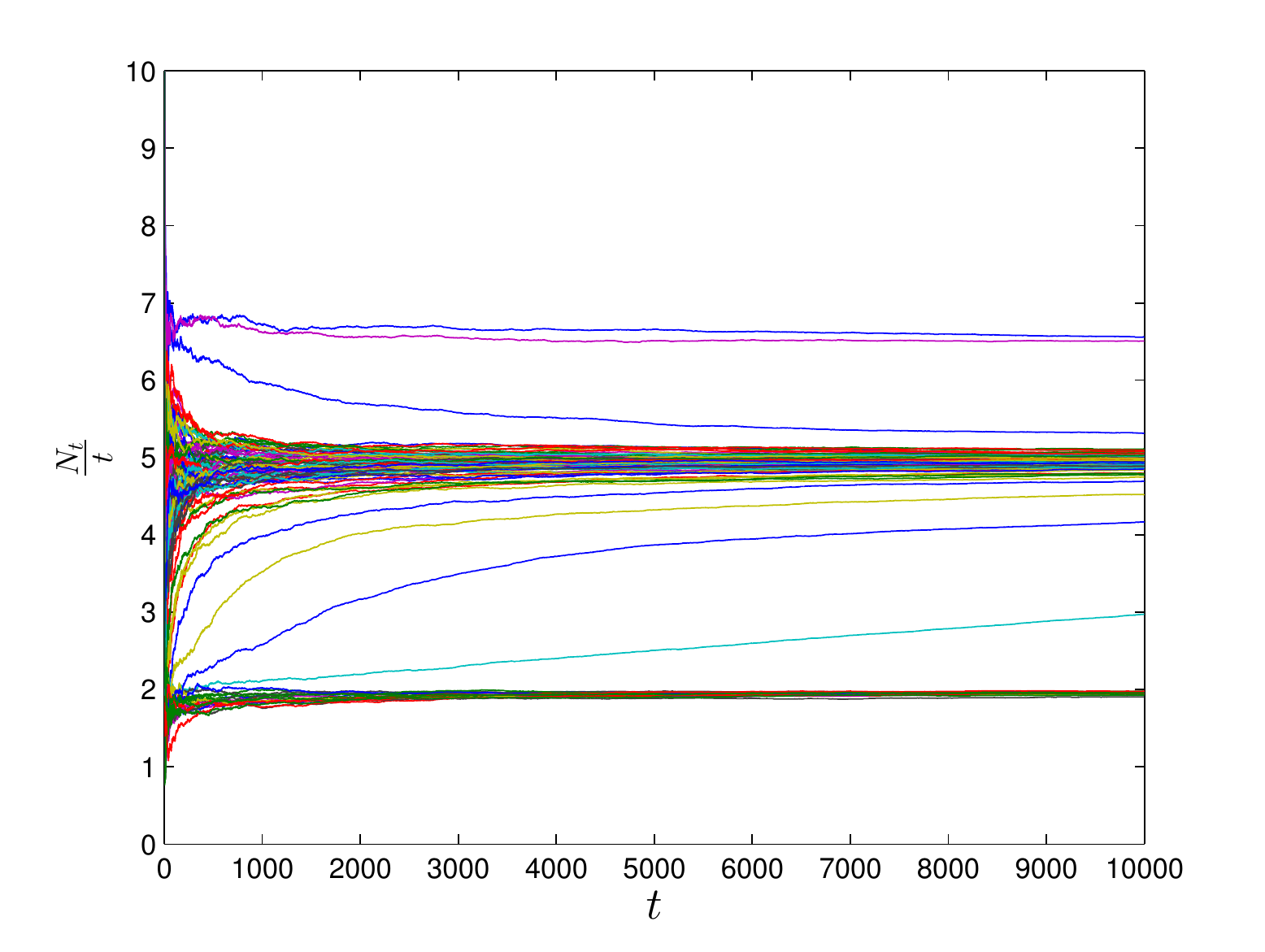}
\caption{We choose the initial starting point as $2.5$. The function $\lambda(x)$ is defined in Figure \ref{lambdaSemi}. 
We simulate 100 sample paths and the illustration suggests
that the limiting set of $\frac{N_{t}}{t}$ as time $t\rightarrow\infty$ is supported on $\{2,5,6.5\}$.}
\label{samplePathSemi}
\end{center}
\end{figure}

\item
Can we relax the assumption $\lambda(\cdot)\leq C_{0}<\infty$ in Theorem \ref{LDPThm} for the large deviations?
Can this assumption be relaxed to $\lim_{x\rightarrow\infty}\frac{\lambda(x)}{x}=0$?
\item
We obtained explicit formulas for the mean, variance and covariance of $N_{t}$ when $\lambda(x)$ is linear (here $x=\lambda(x)$ may not
have a fixed point). 
Can we at least obtain the asymptotics for the mean, variance and covariance for large $t$ when $\lambda(x)$ is nonlinear?
\item
We can also consider a $d$-dimensional simple point process $(N^{(1)}_{t},\ldots,N^{(d)}_{t})$, where $N^{(i)}_{t}$
has intensity at time $t$ given by
\begin{equation*}
\lambda_{t}^{(i)}=\frac{\sum_{j\neq i}a_{ij}N_{t-}^{(j)}}{t+1}+b_{i}.
\end{equation*}
More generally, we can consider for example 
$\lambda_{t}^{(i)}=\lambda_{i}(\frac{1}{t+1}\sum_{j\neq i}a_{ij}N_{t-}^{(j)})$ for nonlinear $\lambda_{i}(\cdot)$.
The $d$-dimensional process $(N^{(1)}_{t},\ldots,N^{(d)}_{t})$ is thus mutually exciting. Can we do the similar analysis to study
the $d$-dimensional process as in our paper?
\end{itemize}

\section*{Acknowledgements}

The authors are grateful to Maury Bramson and Wenqing Hu for helpful discussions.

\end{document}